\documentclass[11pt]{article}

\usepackage{mathrsfs}
\usepackage{amssymb}
\usepackage{amsmath}
\usepackage{amsfonts}
\usepackage[colorlinks,linkcolor=blue,anchorcolor=blue,citecolor=blue]{hyperref}
\usepackage{graphics}
\usepackage{epsfig}
\usepackage{enumerate}

\usepackage{color}

\newtheorem{theorem}{Theorem}[section]
\newtheorem{lemma}[theorem]{Lemma}
\newtheorem{definition}[theorem]{Definition}
\newtheorem{proposition}[theorem]{Proposition}

\newtheorem{remark}[theorem]{Remark}

\newenvironment{proof}{{\bf Proof:}}{~\hfill $\Box$}
\newenvironment{keywords}{{\bf Keywords: }}{}

\allowdisplaybreaks

\numberwithin{equation}{section}
\newcommand{\esssup}{\mathop{\mathrm{esssup}}}
\newcommand{\essinf}{\mathop{\mathrm{essinf}}}

\begin{document}

\title{A BSDE approach to Nash equilibrium payoffs for  stochastic differential games with nonlinear cost
functionals\footnote{ Corresponding address: Laboratoire de Math\'ematiques, CNRS UMR 6205,
Universit\'{e} de Bretagne Occidentale, 6, avenue Victor Le Gorgeu,
CS 93837, 29238 Brest cedex 3, France}}

\author{Qian LIN  $^{1,2}$\footnote{     {\it Email address}:
Qian.Lin@univ-brest.fr}
\\
{\small $^1$School of Mathematics, Shandong University,  Jinan 250100,   China;}\\
{\small $^2$ Laboratoire de Math\'ematiques, CNRS UMR 6205,
Universit\'{e} de Bretagne Occidentale,}\\ {\small 6, avenue Victor
Le Gorgeu, CS 93837, 29238 Brest cedex 3, France.}  }

\date{}
\maketitle

\begin{abstract} In this paper, we study Nash equilibrium payoffs for two-player nonzero-sum
stochastic differential games via the theory of  backward stochastic
differential equations. We obtain an existence theorem and a
characterization theorem of Nash equilibrium payoffs for two-player nonzero-sum
stochastic differential games with nonlinear cost functionals
defined with the help of a doubly controlled backward stochastic
differential equation. Our results extend former ones by Buckdahn,
Cardaliaguet and Rainer \cite{BCR} and are based on a backward
stochastic differential equation approach.
\end{abstract}

\noindent

\section{Introduction}
Since the pioneering work of Isaacs \cite{I}, differential games and
stochastic differential games have been investigated by many
authors.  Fleming and Souganidis \cite{FS} were the first to study
zero-sum stochastic differential games and obtained  that the lower
and the upper value functions of such games satisfy the dynamic
programming principle and coincide under the Isaacs condition.
Recently, basing on the ideas of Fleming and Souganidis \cite{FS},
Buckdahn, Cardaliaguet and Rainer \cite{BCR} studied Nash
equilibrium payoffs for two-player nonzero-sum stochastic differential games,
while Buckdahn and Li \cite{BL2006} generalized at one hand the
 results of Fleming and Souganidis \cite{FS} for
stochastic differential games and simplified the approach
considerably by using backward stochastic differential equations. We
refer the reader to Fleming and Souganidis \cite{FS} and Friedman
\cite{F} for a description of earlier results. In the present paper
we bring ideas of the both papers \cite{BCR} and \cite{BL2006}
together, in order to study Nash equilibrium payoffs for two-player nonzero-sum
stochastic differential games with nonlinear cost functionals.

As concerns deterministic differential games, since the work of
Kononenko \cite{K} in the framework of positional stategies and
Tolwinski, Haurie and Leitmann \cite{THL} in the framework of
Friedman strategies, it is well known that deterministic nonzero-sum
differential games admit Nash equilibrium payoffs. Recently,
Buckdahn, Cardaliaguet and Rainer \cite{BCR} generalized the above
result to two-player nonzero-sum stochastic differential games and obtained an
existence and a characterization for two-player nonzero-sum stochastic
differential games. On the other hand, since the works of Case
\cite{C} and Friedman \cite{F},   Nash equilibrium  payoffs should
be the solution of Hamilton-Jacobi equations. Basing on these ideas,
Bessoussan and Frehse \cite{BF} and Mannucci \cite{M} generalized
the above result to stochastic differential games using the
existence of smooth enough solutions for a system of parabolic
partial differential equations, while Hamad\`{e}ne, Lepeltier and
Peng \cite{HLP}, Hamad\`{e}ne \cite{H} and Lepeltier, Wu and Yu
\cite{LWY} used a saddle point argument in the framework of backward
stochastic differential equations. But both methods rely heavily on
the assumption of the non degeneracy of diffusion coefficients.

In this paper, we investigate Nash equilibrium payoffs for
two-player nonzero-sum stochastic differential games. The generalization of
earlier result by Buckdahn, Cardaliaguet and Rainer \cite{BCR}
concerns the following aspects: Firstly, our cost functionals are
defined by controlled backward stochastic differential equations,
and the admissible control processes depend on events  occurring
before the beginning of the stochastic differential game. Thus, our
cost functionals are not necessarily deterministic. Secondly, since
our cost functionals are nonlinear, we cannot apply the methods used
in Buckdahn, Cardaliaguet and Rainer \cite{BCR}. We make use of the
notion of stochastic backward semigroups  introduced by Peng
\cite{P1997}, and the theory of backward stochastic differential
equations. Finally, each player has his own backward stochastic
differential equation, controlled also by the adversary player,
which defines his own cost functional.

Beyond  the theoretical interest of this paper the result of the
paper is also  applicable in finance and economics. For instance, we
can consider an application of our theoretical result to a problem
arising in financial markets. Let the financial market consist  of a
risk-free asset and  risky stocks and consider  two investors
(players) in this financial market.  Both investors try to maximize
their payoff functionals, which are, in general, different. To
maximize them, they have to use investment strategies with delays.
Indeed, although both investors react immediately to the financial
market, the financial market is not so quick in reacting to the
moves of both investors.  The above described problem  leads to a
two-player nonzero-sum stochastic differential game. We can use our
theoretical result to get an existence theorem and a
characterization theorem of Nash equilibrium payoffs for this game.

 Our paper is organized as follows. In Section 2, we
  introduce some notations and preliminaries concerning backward stochastic differential equations,
    which we will need in what follows.
  In Section 3, we give the main results of this paper and their proofs, i.e., an existence theorem
  and a characterization theorem of Nash equilibrium payoffs for two-player nonzero-sum
stochastic differential games as well as their proofs.

\section{Preliminaries}
Let $(\Omega, \mathcal {F}, \mathbb{P})$ be the classical Wiener
space, i.e., for the given terminal time $T>0$, we consider
$\Omega=C_{0}([0,T];\mathbb{R}^{d})$ as the space of continuous
functions $h:[0,T]\rightarrow \mathbb{R}^{d}$ such that $h(0)=0$,
endowed with the supremum norm, and let $\mathbb{P}$ be the Wiener
measure on the Borel $\sigma$-field $\mathcal {B}(\Omega)$ over
$\Omega$, with respect to which the coordinate process
$B_t(\omega)=\omega_{t}, \omega\in\Omega, t\in[0,T],$ is  a
$d$-dimensional standard Brownian motion. We denote by $\mathcal
{N}_{\mathbb{P}}$ the collection of all $\mathbb{P}$-null sets in
$\Omega$ and define the filtration $\mathbb{F}=\{\mathcal
{F}_{t}\}_{t\in [0,T]}$, which is generated by the coordinate
process $B$ and completed by all $\mathbb{P}$-null sets:
\begin{eqnarray*}
\mathcal {F}_{t}=\sigma\{B_{s},s\leq t\}\vee\mathcal
{N}_{\mathbb{P}}, \ t\in[0,T],
\end{eqnarray*}
where $\mathcal {N}_{\mathbb{P}}$ is the set of all
$\mathbb{P}$-null sets.

 Let us introduce the
following spaces, which will be needed in what follows.
\begin{eqnarray*}
&&\bullet\ L^2 (\Omega, \mathcal {F}_{T}, \mathbb{P};
\mathbb{R}^{n}) = \bigg\{ \xi\ |\  \xi:
\Omega\rightarrow\mathbb{R}^{n}\  \mbox {is an} \ \mathcal {F}_{T}
\mbox {-measurable random variable such
that}\\ &&\qquad \qquad\qquad\qquad\qquad\quad \mathbb{E}[|\xi|^2]<+\infty \bigg\},\\
&&\bullet\ \mathcal {H}^2 (0,T; \mathbb{R}^{d}) =\bigg\{ \varphi\ |\
\varphi:\Omega\times[0, T]\rightarrow\mathbb{R}^{d}\ \mbox {is an} \
\{\mathcal {F}_{t}\}\mbox{-adapted  process such that }\\
&&\qquad \qquad\qquad\qquad\quad\mathbb{E}\int_0^T|\varphi_t|^2dt<+\infty \bigg\},\\
&&\bullet\ S^2 (0,T; \mathbb{R})=\bigg\{ \varphi\ |\
\varphi:\Omega\times[0,
T]\rightarrow\mathbb{R} \ \mbox {is an} \ \{\mathcal {F}_{t}\}\mbox{-adapted continuous process such that}\\
&& \qquad \qquad\qquad\qquad\quad\mathbb{E}[\sup_{0\leq t\leq
T}|\varphi_t|^2]<+\infty \bigg\}.
\end{eqnarray*}
We consider the BSDE with data $(f,\xi)$ :
\begin{equation}\label{equ}
Y_t=\xi+\int_t^T f(s,Y_s,Z_s)ds-\int_t^T Z_sdB_s,\qquad 0\leq t\leq
T.
\end{equation}
Here $f: \Omega \times [0,T]\times \mathbb{R} \times
\mathbb{R}^{d}\rightarrow  \mathbb{R}$ is such that, for any $
(y,z)\in \mathbb{R} \times  \mathbb{R}^{d}$, $f(\cdot,y, z)$ is
progressively measurable. We make the following
  assumptions:

  $(H1)$ (Lipschitz condition): There exists a positive constant $L$  such that
  for all  $ (t,y_{i},z_{i})\in [0,T]\times  \mathbb{R} \times \mathbb{R}^{d},$  $i=1,
  2$,
  $$|f(t,y_{1},z_{1})-f(t,y_{2},z_{2})|\leq L(|y_{1}-y_{2}|+|z_{1}-z_{2}|).$$

  $(H2)$ $f (\cdot,0, 0)\in \mathcal {H}^2 (0,T; \mathbb{R})$.\\

The following existence  and uniqueness theorem was established by
Pardoux and Peng \cite{PP1990}.
\begin{lemma}
Let the assumptions $(H1)$ and $(H2)$ hold. Then, for all $\xi\in
L^2 (\Omega, \mathcal {F}_{T}, \mathbb{P};\mathbb{R})$,  BSDE
(\ref{equ}) has a unique solution
\begin{eqnarray*}
(Y,Z)\in S^2 (0,T; \mathbb{R})\times\mathcal{H}^2 (0,T;
\mathbb{R}^{d}).
\end{eqnarray*}
\end{lemma}

We recall the well-known comparison theorem for solutions of BSDEs,
which has been established by El Karoui, Peng and Quenez
\cite{KPQ1997} and Peng \cite{P1997}.

\begin{lemma}\label{l1}
Let  $\xi^{1}, \xi^{2}\in L^2 (\Omega, \mathcal {F}_{T},
\mathbb{P};\mathbb{R})$, and $f^{1}$ and $f^{2}$ satisfy $(H1)$ and
$(H2)$. We denote by
 $(Y^{1},Z^{1})$ and $(Y^{2},Z^{2})$  the solutions of BSDEs
with data $(f^{1},\xi^{1})$ and $(f^{2},\xi^{2})$,
 respectively, and we suppose that

  (i) $\xi^{1}\leq \xi^{2}$, $\mathbb{P}-a.s.,$

  (ii)  $f^{1}(t,Y_{t}^{2}, Z_{t}^{2}) \leq f^{2}(t, Y_{t}^{2}, Z_{t}^{2})$,
  $dtd\mathbb{P}-a.e.$\\
Then, we have $Y_{t}^{1} \leq Y_{t}^{2}$, $ a.s.$, for all $t \in
[0,T]$.  Moreover, if $\mathbb{P}(\xi^{1}< \xi^{2})>0,$ then
$\mathbb{P}(Y_{t}^{1} < Y_{t}^{2})>0, t\in[0,T],$ and in particular,
$Y_{0}^{1} < Y_{0}^{2}$.
\end{lemma}

By virtue of  the notations introduced in the above Lemma,  for some
$f: \Omega \times [0,T]\times \mathbb{R} \times
\mathbb{R}^{d}\rightarrow \mathbb{R}$, we put
$$f^{1}(s,y,z)=f(s,y,z)+\varphi_{1}(s),\ f^{2}(s,y,z)=f(s,y,z)+\varphi_{2}(s).$$
Then we have the following lemma. For the proof the readers can
refer to El Karoui, Peng and Quenez \cite{KPQ1997}, and Peng
\cite{P1997}.

\begin{lemma}\label{l3}
Suppose that $\xi^{1}, \xi^{2}\in L^2 (\Omega, \mathcal {F}_{T},
\mathbb{P})$, $f$ satisfies $(H1)$ and $(H2)$ and
$\varphi_{1},\varphi_{2}\in\mathcal {H}^2 (0,T; \mathbb{R})$. We
denote by
 $(Y^{1},Z^{1})$ and $(Y^{2},Z^{2})$  the solution of BSDEs
 (\ref{equ}) with data $(f^{1},\xi^{1})$ and $(f^{2},\xi^{2})$,
 respectively. Then we have the following estimate:
\begin{eqnarray*}
&&|Y_{t}^{1}-Y_{t}^{2}|^{2} +\dfrac{1}{2}\mathbb{E}\left\{\int_t^T
e^{\beta(t-s)}[|Y_{s}^{1}-Y_{s}^{2}|^{2}+|Z_{s}^{1}-Z_{s}^{2}|^{2}]ds\
 \Big |\ \mathcal {F}_t\right\}\\
&\leq& \mathbb{E}\left\{ e^{\beta(T-t)}|\xi^{1}-\xi^{2}|^{2}\
 \Big |\ \mathcal {F}_t\right\}+\mathbb{E}\left\{\int_t^T
e^{\beta(t-s)}|\varphi_{1}(s)-\varphi_{2}(s)|^{2}ds\
 \Big |\ \mathcal {F}_t\right\},
\end{eqnarray*}
where $\beta=16(1+L^{2})$ and $L$ is the Lipschitz constant in
$(H1)$.
\end{lemma}

\section{Nash equilibrium payoffs for  nonzero-sum stochastic differential games }
The objective of this section is to investigate Nash equilibrium
payoffs for two-player nonzero-sum stochastic differential games  with
nonlinear cost functionals. An existence theorem (Theorem \ref{t2})
and a characterization theorem (Theorem \ref{t1}) of Nash
equilibrium payoffs for two-player nonzero-sum stochastic differential games
are the main results of this section.

Let $U$ and  $V$ be two compact metric spaces.  Here $U$ is
considered as the control state space of the first player, and $V$
as that of the second one. The associated sets of admissible
controls will be denoted by $\mathcal {U}$ and $\mathcal {V}$,
respectively. The set $\mathcal {U}$ is formed by all $U$-valued
$\mathbb{F}$-progressively measurable processes,  and $\mathcal {V}$
is the set of all $V$-valued $\mathbb{F}$-progressively measurable
processes.

For  given admissible controls $u(\cdot)\in\mathcal {U}$ and
$v(\cdot)\in\mathcal {V}$, we consider the following control system
\begin{equation}\label{eq1}
\left\{
\begin{array}{rcl}
dX^{t,\zeta;u,v}_s &=& b(s,X^{t,\zeta;u,v}_s,
u_s,v_s)ds+\sigma(s,X^{t,\zeta;u,v}_s,u_s,v_s)dB_s,\qquad
s\in [t,T],\\
X^{t,\zeta;u,v}_t &=& \zeta,
\end{array}
\right.
\end{equation}
where $t\in [0,T]$ is regarded as the initial time, and $\zeta\in
L^2(\Omega,\mathcal {F}_t,\mathbb{P};\mathbb{R}^n)$ as the initial
state. The mappings
\[
b:[0,T]\times\mathbb{R}^n\times U\times
V\rightarrow\mathbb{R}^n,\qquad \sigma:[0,T]\times\mathbb{R}^n\times
U\times V\rightarrow\mathbb{R}^{n\times d}
\]
are supposed to satisfy the following conditions:

\begin{list}{}{\setlength{\itemindent}{0cm}}
\item[$(H3.1)$] For all $x\in\mathbb{R}^{n}$, $b(\cdot,x,\cdot,\cdot)$ and $\sigma(\cdot,x,\cdot,\cdot)$
are continuous in $(t,u,v)$;
\item[$(H3.2)$] There exists a positive constant $L$ such that, for all $t\in [0,T], x,x^{\prime}\in\mathbb{R}^n$,
$u\in U, v\in V$,
\[
|b(t,x,u,v)-b(t,x^{\prime},u,v)|+|\sigma(t,x,u,v)-\sigma(t,x^{\prime},u,v)|\leq
L|x-x^{\prime}|.
\]
\end{list}
It is obvious that, under the above conditions, for any $u(\cdot)\in
\mathcal {U}$ and $v(\cdot)\in \mathcal {V}$, the control system
(\ref{eq1}) has a unique strong solution $\{X^{t,\zeta;u,v}_s,\ 0\le
t\le s\le T\}$, and we also have the following estimates.

\begin{lemma}\label{l4}
For all $p\geq 2$, there exists a  positive constant $C_{p}$ such
that, for all $t\in [0,T]$, $\zeta,\zeta^{\prime}\in
L^2(\Omega,\mathcal {F}_t,\mathbb{P};\mathbb{R}^n)$,
$u(\cdot)\in\mathcal {U}$ and $v(\cdot)\in \mathcal {V}$,
\begin{eqnarray*}
&&\mathbb{E}\left\{\sup_{t\leq s\leq
T}|X^{t,\zeta;u,v}_s|^p\Big|\mathcal {F}_t\right\}\leq
C_{p}(1+|\zeta|^p),\quad \mathbb{P}-a.s.,\\
&& \mathbb{E}\left\{\sup_{t\leq s\leq
T}|X^{t,\zeta;u,v}_s-X^{t,\zeta';u,v}_s|^p\Big|\mathcal
{F}_t\right\} \leq C_{p}|\zeta-\zeta^{\prime}|^p ,\quad
\mathbb{P}-a.s.
\end{eqnarray*}
Here the constant $C_{p}$ only depends  on $p$, the Lipschitz
constant and the linear growth of $b$ and $\sigma$.
\end{lemma}

For  given admissible controls $u(\cdot)\in\mathcal {U}$ and
$v(\cdot)\in\mathcal {V}$, we consider the following  BSDE:
\begin{equation}\label{BSDE}
\begin{array}{lll}
Y^{t,\zeta;u,v}_s &=& \Phi(X^{t,\zeta;u,v}_T)
+\displaystyle \int_s^T f(r,X^{t,\zeta;u,v}_r,Y^{t,\zeta;u,v}_r,Z^{t,\zeta;u,v}_r,u_r, v_r)dr\\
&&-\displaystyle \int_s^TZ^{t,\zeta;u,v}_rdB_r,\qquad t\leq s\leq T,
\end{array}
\end{equation}
where $X^{t,\zeta;u,v}$ is introduced in equation (\ref{eq1}) and
\begin{eqnarray*}
 \Phi=\Phi(x):\mathbb{R}^n\rightarrow\mathbb{R},
\quad
f=f(t,x,y,z,u,v):[0,T]\times\mathbb{R}^n\times\mathbb{R}\times\mathbb{R}^d\times
U\times V\rightarrow \mathbb{R}
\end{eqnarray*}
satisfy the following conditions:

\begin{list}{}{\setlength{\itemindent}{0cm}}
\item[($H3.3$)] For all $(x,y,z)\in\mathbb{R}^{n}\times\mathbb{R}\times\mathbb{R}^{d}$,
 $f(\cdot,x,y,z,\cdot,\cdot)$ is continuous in $(t,u,v)$;
\item[$(H3.4)$] There exists a positive constant $L$ such that, for all $t\in[0,T], x,x^{\prime}\in\mathbb{R}^n$,
$y,y^{\prime}\in\mathbb{R}$, $z,z^{\prime}\in\mathbb{R}^d$, $u\in U$
and $ v\in V$,
\begin{eqnarray*}
&&|f(t,x,y,z,u,v)-f(t,x^{\prime},y^{\prime},z^{\prime},u,v)|
+|\Phi(x)-\Phi(x^{\prime})|\\
&&\leq L(|x-x^{\prime}|+|y-y^{\prime}|+|z-z^{\prime}|).
\end{eqnarray*}
\end{list}
It is by now standard that under the above assumptions equation
(\ref{BSDE}) admits a unique solution
$(Y^{t,\zeta;u,v},Z^{t,\zeta;u,v})\in S^2 (0,T;
\mathbb{R})\times\mathcal{H}^2 (0,T; \mathbb{R}^{d})$. Moreover, in
Buckdahn and Li \cite{BL2006} it was shown that the following holds:

\begin{proposition}\label{p3}
There exists a  positive constant $C$ such that, for all $t\in
[0,T]$, $u(\cdot)\in\mathcal {U}$ and $v(\cdot)\in \mathcal {V}$,
$\zeta,\zeta^{\prime}\in L^2(\Omega,\mathcal
{F}_t,\mathbb{P};\mathbb{R}^n)$,
\begin{eqnarray*}
&&|Y^{t,\zeta;u,v}_t|\leq C(1+|\zeta|), \quad \mathbb{P}-a.s.,\\
&&|Y^{t,\zeta;u,v}_t-Y^{t,\zeta^{\prime};u,v}_{t}| \leq
C|\zeta-\zeta^{\prime}|,\quad \mathbb{P}-a.s.
\end{eqnarray*}
\end{proposition}

We now introduce  subspaces of admissible controls and give the
definition of admissible strategies.
\begin{definition} The space $\mathcal {U}_{t,T}$ (resp. $\mathcal {V}_{t,T}$) of admissible
controls for Player I (resp., II) on the interval $[t, T]$  is
defined as the space of all processes   $\{u_{r}\}_{r\in[t,T]}$
(resp., $\{v_{r}\}_{r\in[t,T]}$), which are
$\mathbb{F}$-progressively measurable and take values in $U$ (resp.,
$V$).
\end{definition}

\begin{definition}\label{d1}
 A nonanticipating strategy with delay (NAD strategy)
for Player I is a measurable mapping $\alpha:\mathcal
{V}_{t,T}\rightarrow \mathcal {U}_{t,T}$, which satisfies the
following properties:

1)  $\alpha$ is a nonanticipative strategy, i.e., for every
$\mathbb{F}$-stopping time $\tau:\Omega\rightarrow [t,T],$ and for
 $v_{1},v_{2} \in\mathcal {V}_{t,T}$ with $v_{1}=v_{2}$   on $[[ t,\tau]]$,
it holds $\alpha(v_{1})=\alpha(v_{2})$ on $[[ t,\tau]]$. (Recall
that $[[ t,\tau]]=\{(s,\omega)\in[t,T]\times \Omega, t\leq
s\leq\tau(\omega) \}$).

2) $\alpha$ is a nonanticipative strategy with delay, i.e., for all
$v\in\mathcal {V}_{t,T}$,
 there exists an increasing sequence of stopping times
$\{S_{n}(v)\}_{n\geq 1}$ with

i) $t=S_{0}(v)\leq S_{1}(v) \leq\cdots
\leq S_{n}(v)\leq \cdots \leq T$,

ii) $ S_{n}(v) < S_{n+1}(v)$ on $\{S_{n}(v)<T\}, n\geq 0,$

iii) $\mathbb{P}(\bigcup_{n\geq
1}\{S_{n}(v)=T\})=1$, \\
such that, for all $n\geq 1 $, $\Lambda\in\mathcal {F}_{t}$ and
$v,v' \in \mathcal {V}_{t,T}$, we have:  if $v=v'$
 on $[[ t,S_{n-1}(v)]]\bigcap (\Lambda\times[t,T])$, then

iv) $S_{l}(v)=S_{l}(v')$, on $\Lambda$, $\mathbb{P}$-a.s., $1\leq l
\leq n$,

v) $\alpha(v)=\alpha(v')$, on $[[ t,S_{n}(v)]]\bigcap
(\Lambda\times[t,T])$.\vskip1mm

The set of all NAD strategies for Player I for  games over the time
interval $[t,T]$ is denoted by $\mathcal {A}_{t,T}$. The set of all
NAD strategies $\beta:\mathcal {U}_{t,T}\rightarrow \mathcal
{V}_{t,T}$ for Player II for games over the time interval $[t,T]$ is
defined symmetrically and denoted by $\mathcal {B}_{t,T}$.
\end{definition}

We have the following lemma, which is useful in what follows.
\begin{lemma}\label{l2}
Let $\alpha\in \mathcal {A}_{t,T}$ and $\beta\in \mathcal {B}_{t,T}$.
 Then there exists a unique couple of admissible control processes $(u,v)\in
\mathcal {U}_{t,T} \times \mathcal {V}_{t,T}$ such that
\begin{eqnarray*}\label{}
\alpha(v)=u,\quad \beta(u)=v.
\end{eqnarray*}
\end{lemma}
Such a result can be found already in \cite{BCR}. However, since our
definition of NAD strategies differs, we shall provide its proof.

For given  control processes $u(\cdot)\in\mathcal {U}_{t,T}$ and $v(\cdot)\in\mathcal {V}_{t,T}$, we
introduce now the associated cost functional
\begin{equation*}\label{}
J(t,x;u,v):= \left.Y^{t,x;u,v}_s\right|_{s=t},\qquad (t,x)\in
[0,T]\times\mathbb{R}^n.
\end{equation*}
(Recall that $Y^{t,x;u,v}$ is defined by BSDE (\ref{BSDE}) with
$\zeta=x\in\mathbb{R}^{n}$). We define the lower and the upper value
functions $W$ and $U$, resp., of the game: For all $(t,x)\in
[0,T]\times\mathbb{R}^n$, we put
\begin{eqnarray*}\label{}
W(t,x):=\esssup_{\alpha\in\mathcal {A}_{t,T}}
\essinf_{\beta\in\mathcal {B}_{t,T}} J(t,x;\alpha,\beta),
\end{eqnarray*}
and
\begin{eqnarray*}\label{}
U(t,x):=\essinf_{\beta\in\mathcal {B}_{t,T}}
\esssup_{\alpha\in\mathcal {A}_{t,T}} J(t,x;\alpha,\beta).
\end{eqnarray*}
Here we use Lemma \ref{l2} to identify $(X^{t,x;\alpha,\beta},Y^{t,x;\alpha,\beta},Z^{t,x;\alpha,\beta})=(X^{t,x;u,v},Y^{t,x;u,v},Z^{t,x;u,v})$,
and, in particular,
$J(t,x;\alpha,\beta)=J(t,x;u,v)$, where $(u,v)\in\mathcal
{U}_{t,T}\times\mathcal {V}_{t,T}$ is the couple of controls
associated with $(\alpha,\beta)\in\mathcal {A}_{t,T}\times \mathcal
{B}_{t,T}$ by the relation $(\alpha(v),\beta(u))=(u,v)$.

\begin{remark}
 For the convenience of the reader we recall
the notion of the essential infimum and  the essential supremum of
families of random variables (see, e.g., Karatzas and Shreve
\cite{KS1998} for more details). Given a family of  $\mathcal
{F}$-measurable  real valued random variables $\xi_{\alpha}$
($\alpha\in I$), an $\mathcal {F}$-measurable random variable $\xi$
is said to be $\essinf_{\alpha \in I}\xi_{\alpha}$, if

(i) $\xi\leq \xi_{\alpha}, \mathbb{P}-a.s.,$ for all $\alpha \in I$;

(ii) if for any  random variable $\eta$ such that
$\eta\leq\xi_{\alpha}$, $\mathbb{P}$-a.s.,  for all $\alpha \in I$,
it holds that  $\eta\leq\xi$, $\mathbb{P}$-a.s.

We introduce notion of  $\esssup_{\alpha \in I}\xi_{\alpha}$ by the
following relation:
\begin{eqnarray*}\label{}
\esssup_{\alpha \in I}\xi_{\alpha}=-\essinf_{\alpha \in
I}(-\xi_{\alpha}).
\end{eqnarray*}
\end{remark}

\begin{remark}
 Lemma \ref{l2} guarantees that for NAD strategies  $\alpha\in
\mathcal {A}_{t,T}$ and $\beta\in \mathcal {B}_{t,T}$  there exists
a unique associate couple $(u,v)\in \mathcal {U}_{t,T} \times
\mathcal {V}_{t,T}$ of admissible controls  such that $\alpha(v)=u,
\beta(u)=v.$ For general nonanticipative strategies we can, in
general, not get such a couple of controls. Let us give an example:
We suppose that $U=V$ and $\varphi,\psi:U\rightarrow U$ are
measurable functions such that $\psi\circ\varphi$ doesn't have a
fixed point. We define
\begin{eqnarray*}\label{}
\alpha(v)_{s}=\varphi(v_{s}), s\in [t,T], v\in \mathcal {V}_{t,T},\\
\beta(u)_{s}=\psi(u_{s}), s\in [t,T], u\in \mathcal {U}_{t,T}.
\end{eqnarray*}
Then  $\alpha$ and $\beta$ are nonanticipative strategies for Player
I and II, respectively. But there is no couple $(u,v)\in \mathcal
{U}_{t,T} \times \mathcal {V}_{t,T}$ such that  $\alpha(v)=u,
\beta(u)=v.$ Indeed, if there existed such couple of controls
$(u,v)\in \mathcal {U}_{t,T} \times \mathcal {V}_{t,T}$, we would
have, for  $ s\in [t,T],$
\begin{eqnarray*}\label{}
 & u_{s}=\alpha(v)_{s}=\varphi(v_{s}),\\
&v_{s}=\beta(u)_{s}=\psi(u_{s})= \psi\circ\varphi(v_{s}).
\end{eqnarray*}
\end{remark}
But this means that $v_{s}$ is a fixed point of $\psi\circ\varphi$,
which contradicts the assumptions of the absence of fixed
points.\vskip2mm

Let us  now give the proof of Lemma \ref{l2}.\vskip2mm

\begin{proof}
We give the proof in two steps.\vskip1mm

{\bf Step 1}: For $(u,v)\in\mathcal {U}_{t,T} \times \mathcal
{V}_{t,T}$ we denote by $\{S_{n}(v)\}_{n\geq 1}$ (resp.
$\{T_{n}(u)\}_{n\geq 1}$) the sequence of stopping times associated
with $\alpha\in \mathcal {A}_{t,T}$ (resp.  $\beta\in \mathcal
{B}_{t,T}$) by Definition \ref{d1}.  Then, for arbitrarily given
$(u,v)\in\mathcal {U}_{t,T} \times \mathcal {V}_{t,T}$  we define
the optional set
\begin{eqnarray*}\label{}
\Gamma:=\bigcup_{n\geq1}\Big([S_{n}(v)]\cup [T_{n}(u)]\Big),
\end{eqnarray*}
where $[S_{n}(v)]$ (resp. $[T_{n}(u)]$) denotes the graph of
$S_{n}(v)$ (resp. $T_{n}(u)$). Then, for $\omega \in \Omega$, we
have
\begin{eqnarray*}\label{}
\Gamma(\omega)=\Big\{S_{n}(v)(\omega), T_{l}(u)(\omega), n,l\geq 1,
s.t.\ S_{n}(v)(\omega)<T, T_{l}(u)(\omega)<T\Big \}\bigcup \big\{
T\big \},
\end{eqnarray*}
and we observe that $\Gamma(\omega)$ is a finite set.

We denote by $D_{\Gamma}$ the first hitting time of $\Gamma$,  and
we define a sequence of $\{\mathcal {F}_{r}\}$-stopping times as
follows:
\begin{eqnarray*}\label{}
\tau_{0}&=&t,\\
\tau_{1}(u,v)&=&D_{\Gamma}(=S_{1}(v)\wedge T_{1}(u)),\\
\tau_{2}(u,v)&=&D_{\Gamma\setminus{[\tau_{1}(u,v)]}}\wedge T,\\
&\vdots&\\
\tau_{n}(u,v)&=&D_{\Gamma\setminus{\cup_{i=1}^{n-1}[\tau_{i}(u,v)]}}\wedge
T, n\geq 1.
\end{eqnarray*}
Recall that $a\wedge b=\min\{a,b\}, a,b\in\mathbb{R}$.

We notice that $\tau_{1}(u,v)$ is independent of $(u,v)$, and for
$n\geq2$, $\tau_{n}(u,v)$ depends only on $(u,v)|_{[[
t,\tau_{n-1}(u,v)]]}$. Indeed, this is a direct consequence of point
2) in Definition \ref{d1} and the  definition of
$\{\tau_{n}(u,v)\}_{n\geq 1}$.

 From the definition of
$\{\tau_{n}\}_{n\geq 1}$ it follows that, for all $(u,v)\in\mathcal
{U}_{t,T} \times \mathcal {V}_{t,T}$,

i) $t=\tau_{0}\leq \tau_{1}(u,v) \leq\cdots
\leq \tau_{n}(u,v)\leq \cdots \leq T$,

ii) $ \tau_{n}(u,v) < \tau_{n+1}(u,v),$ on $\{\tau_{n}(u,v)<T\},
n\geq 0.$  Moreover, since $\Gamma(\omega)$ is a finite set,
$\mathbb{P}(d\omega)-a.s.,$   $\mathbb{P}(\bigcup_{n\geq
1}\{\tau_{n}(u,v)=T\})=1$.

iii) For $n\geq 1 $ and  all $(u,v), (u',v')\in \mathcal {U}_{t,T}
\times \mathcal {V}_{t,T}$, it holds: if $(u,v)=(u',v')$
 on $[[ t,\tau_{n-1}(u,v)]]$, then
 $\tau_{l}(u,v)=\tau_{l}(u',v')$, $1\leq l \leq n$,
and $\alpha(v)=\alpha(v')$ and $\beta(u)=\beta(u')$, on $[[ t,\tau_{n}(u,v)]]$.\\

{\bf Step 2}: For $\alpha\in \mathcal {A}_{t,T}$ and $\beta\in
\mathcal {B}_{t,T}$, we let  $\{\tau_{n}\}_{n\geq 1}$ be constructed
as above. Since neither $\tau_{1}$ depends on the controls nor
$(\alpha,\beta)$ restricted to  $[[t,\tau_{1}]]$ does, the process
\begin{eqnarray*}\label{}
(u^{0},v^{0}):=(\alpha(v_{0}),\beta(u_{0})), \ \text{for arbitrary}\
(u_{0},v_{0})\in \mathcal {U}_{t,T} \times \mathcal {V}_{t,T},
\end{eqnarray*}
is such that $\alpha(v^{0})=u^{0}$ and $\beta(u^{0})=v^{0}$, on
$[[t,\tau_{1}]]$.

Taking into account that $\tau_{2}$ only depends on the controls
restricted to  $[[ t,\tau_{1}]]$, and $(\alpha(v^{0}),\beta(u^{0}))$
$|_{[[ t,\tau_{2}(u^{0},v^{0})]]}$ only depends on the controls
$(u^{0},v^{0})$ restricted to  $[[t,\tau_{1}]]$, we can define
\begin{eqnarray*}\label{}
(u^{1},v^{1}):=(\alpha(v^{0}),\beta(u^{0})),
\end{eqnarray*}
and  since we have $(u^{1},v^{1})=(u^{0},v^{0})$ on on $[[
t,\tau_{1}]]$,
 it follows that $(u^{1},v^{1})=(\alpha(v^{1}),\beta(u^{1}))$, on $[[ t,\tau_{2}(u^{1},v^{1})]]$.
 Repeating the above argument we put
\begin{eqnarray*}\label{}
(u^{n},v^{n}):=(\alpha(v^{n-1}),\beta(u^{n-1}))\in \mathcal
{U}_{t,T} \times \mathcal {V}_{t,T}.
\end{eqnarray*}
Then, since due to $(n-1)$th iteration step
$(\alpha(v^{n-1}),\beta(u^{n-1}))=(u^{n-1},v^{n-1})$, on $[[
t,\tau_{n}(u^{n-1},v^{n-1})]]$,
 we also have $(u^{n-1},v^{n-1})=(u^{n},v^{n})$, on $[[
t,\tau_{n}(u^{n-1},v^{n-1})]]$, and, thus, also
$\tau_{l}(u^{n},v^{n})=\tau_{l}(u^{n-1},v^{n-1}),$  $ 0\leq l \leq
n+1$. Hence,
\begin{eqnarray*}\label{}
\tau_{n}(u^{n-1},v^{n-1})=\tau_{n}(u^{n},v^{n})\leq\tau_{n+1}(u^{n},v^{n})=\tau_{n+1}(u^{n+1},v^{n+1}),
n\geq 1,
\end{eqnarray*}
from which we deduce the existence of the limit of stopping times
\begin{eqnarray*}\label{}
\tau:=\lim\limits_{n\rightarrow\infty}\tau_{n}(u^{n},v^{n})\leq T.
\end{eqnarray*}
For arbitrarily given  $(u_{0},v_{0})\in \mathcal {U}_{t,T} \times
\mathcal {V}_{t,T}$ we define
\begin{eqnarray*}\label{}
(u,v):=\sum\limits_{n\geq
0}(u^{n},v^{n})1_{[[\tau_{n-1}(u^{n-1},v^{n-1}),\tau_{n}(u^{n},v^{n})[[}
+(u_{0},v_{0})1_{[[\tau,T]]}.
\end{eqnarray*}
Obviously, $(u,v)\in \mathcal {U}_{t,T} \times \mathcal {V}_{t,T}$,
and since $(u,v)=(u^{n},v^{n})$ on $[[ t,\tau_{n}(u^{n},v^{n})]]$,
we have $\tau_{l}(u,v)=\tau_{l}(u^{n},v^{n}),$  $ 0\leq l \leq n+1,
n\geq 0$ (see the above property iii)). But this allows to conclude
from ii) that
\begin{eqnarray*}\label{}
\mathbb{P}\big(\bigcup_{n\geq
1}\{\tau_{n}(u^{n},v^{n})=T\}\big)=\mathbb{P}\big(\bigcup_{n\geq
1}\{\tau_{n}(u,v)=T\}\big)=1.
\end{eqnarray*}
Consequently, since the above defined process $(u,v)\in \mathcal
{U}_{t,T} \times \mathcal {V}_{t,T}$ has the property that
\begin{eqnarray*}\label{}
(\alpha(v),\beta(u))&=&(\alpha(v^{n}),\beta(u^{n}))\ \text{on} \ [[
t,\tau_{n}(u,v)]]\  \ (\text{nonanticipativity of}\
(\alpha, \beta))\\
&=&(u^{n},v^{n})=(u,v)\ \text{on} \ [[ t,\tau_{n}(u,v)]],
\end{eqnarray*}
we have $(\alpha(v),\beta(u))=(u,v)$ on $[t,T]\times \Omega,
dsd\mathbb{P}-a.e.$ The proof is complete.
\end{proof}\\

The following lemmas were established in \cite{BCQ} under a slightly
different definition of NAD strategies. However, their validity in
our new framework can be checked easily.

\begin{lemma}
Under assumptions (H3.1)--(H3.4), for all
$(t,x)\in[0,T]\times\mathbb{R}^n$, the value functions $W(t,x)$ and
$ U(t,x)$ are deterministic functions.
\end{lemma}

\begin{lemma} \label{le5}
There exists a positive constant $C$ such that, for all $t,t'\in [0,T]$ and $x, x^{\prime}\in\mathbb{R}^n$, we have
\begin{enumerate}[(i)]
\item $W(t,x)$ is $\frac{1}{2}$-H\"{o}lder continuous in $t$:
$$|W(t,x)-W(t',x)|\leq C(1+|x|)|t-t^{\prime}|^{\frac{1}{2}};$$
\item $|W(t,x)-W(t,x^{\prime})|\leq C|x-x^{\prime}|.$
\end{enumerate}
The same properties hold true for the function $U$.
\end{lemma}

\begin{remark}
From the above Lemma it follows, in particular, that the functions
$W$ and $U$ are of at most linear growth, i.e., there exists a
positive constant $C$ such that, for all $t\in [0,T]$ and
$x\in\mathbb{R}^n$,
 $|W(t,x)|\leq C(1+|x|)$.
\end{remark}

We now recall the notion of stochastic backward semigroups, which
was introduced by Peng \cite{P1997} and translated by Buckdahn and
Li \cite{BL2006} into the framework of stochastic differential
games. For a given initial condition $(t,x)$, a positive number
$\delta\leq T-t$, for admissible control processes
$u(\cdot)\in\mathcal {U}_{t,t+\delta}$ and $v(\cdot)\in\mathcal
{V}_{t,t+\delta}$, and a real-valued random variable $\eta\in
L^2(\Omega,\mathcal {F}_{t+\delta},\mathbb{P};\mathbb{R})$, we
define
\[
G^{t,x;u,v}_{t,t+\delta}[\eta]:=\overline{Y}_t^{t,x;u,v},
\]
where $(\overline{Y}_s^{t,x;u,v},\overline{Z}_s^{t,x;u,v})_{t\leq
s\leq t+\delta}$ is the unique solution of the following BSDE with
time horizon $t+\delta$:
\begin{eqnarray*}
 \overline{Y}_s^{t,x;u,v}&=&\eta+\int_s^{t+\delta}f(r,X^{t,x;u,v}_r,\overline{Y}_r^{t,x;u,v},
 \overline{Z}_r^{t,x;u,v},u_r,v_r)dr\\
&& -\int_s^{t+\delta}\overline{Z}_r^{t,x;u,v}dB_r,\qquad t\leq s\leq
t+\delta,
\end{eqnarray*}
and $X^{t,x;u,v}$ is the unique solution of equation (\ref{eq1})
with $\zeta=x\in\mathbb{R}^{n}$.

We observe that for the solution $Y^{t,x;u,v}$ of BSDE (\ref{BSDE}) with $\zeta=x\in\mathbb{R}^{n}$ we have
\begin{eqnarray*}\label{}
J(t,x;u,v)&=&
Y^{t,x;u,v}_t=G^{t,x;u,v}_{t,T}[\Phi(X^{t,x;u,v}_T)]=G^{t,x;u,v}_{t,t+\delta}[Y^{t,x;u,v}_{t+\delta}]\\
&=&G^{t,x;u,v}_{t,t+\delta}[J(t+\delta,X^{t,x;u,v}_{t+\delta};u,v)].
\end{eqnarray*}

\begin{remark}\label{remark}
For the special case that  $f$ is independent of $(y,z)$  we have
\begin{eqnarray*}
G^{t,x;u,v}_{s,T}[\Phi(X^{t,x;u,v}_T)]&=&G^{t,x;u,v}_{s,t+\delta}[Y^{t,x;u,v}_{t+\delta}]\\
&=&\mathbb{E}[Y^{t,x;u,v}_{t+\delta}+\int_s^{t+\delta}f(r,X^{t,x;u,v}_r,u_r,v_r)dr\
\Big |\ \mathcal {F}_{s}],\ s\in [t,t+\delta].
\end{eqnarray*}
In particular,
\begin{eqnarray*}
G^{t,x;u,v}_{t,T}[\Phi(X^{t,x;u,v}_T)]
=\mathbb{E}[\Phi(X^{t,x;u,v}_T)+\int_t^{T}f(r,X^{t,x;u,v}_r,u_r,v_r)dr\
\Big |\ \mathcal {F}_{t}].
\end{eqnarray*}
\end{remark}

For more details on stochastic backward semigroups the reader is
referred to  Peng \cite{P1997} and Buckdahn and Li \cite{BL2006}.
Let us also recall the following dynamic programming principle for
the value functions of stochastic differential games. Its proof can
be found in  \cite{BCQ}.
\begin{proposition}\label{p1}
Under the assumptions (H3.1)--(H3.4)  the following dynamic programming
principle holds: for all $0<\delta\leq T-t, x\in\mathbb{R}^{n}$,
\begin{equation*}
W(t,x)=\esssup_{\alpha\in\mathcal
{A}_{t,t+\delta}}\essinf_{\beta\in\mathcal {B}_{t,t+\delta}}
G^{t,x;\alpha,\beta}_{t,t+\delta}[W(t+\delta,X^{t,x;\alpha,\beta}_{t+\delta})],
\end{equation*}
and
\begin{equation*}
U(t,x)=\essinf_{\beta\in\mathcal
{B}_{t,t+\delta}}\esssup_{\alpha\in\mathcal {A}_{t,t+\delta}}
G^{t,x;\alpha,\beta}_{t,t+\delta}[U(t+\delta,X^{t,x;\alpha,\beta}_{t+\delta})].
\end{equation*}
\end{proposition}

After having recalled some basics on two-player zero-sum stochastic
differential games, let us introduce the framework of two-player nonzero-sum
stochastic differential games where each of the both players has his
own terminal as well as running cost functionals $\Phi_{j}$ and
$f_{j}$, respectively, $j=1,2$. More precisely, for arbitrarily
given admissible controls $u(\cdot)\in\mathcal {U}$ and
$v(\cdot)\in\mathcal {V}$, we consider the following  BSDEs, $
j=1,2,$
\begin{equation*}\label{}
\begin{array}{lll}
^{j}Y^{t,\zeta;u,v}_s &=& \Phi_{j}(X^{t,\zeta;u,v}_T)
+\displaystyle \int_s^T f_{j}(r,X^{t,\zeta;u,v}_r,\ ^{j}Y^{t,\zeta;u,v}_r,\ ^{j}Z^{t,\zeta;u,v}_r,u_r, v_r)dr\\
&&-\displaystyle \int_s^T\ ^{j}Z^{t,\zeta;u,v}_rdB_r,\qquad t\leq
s\leq T,
\end{array}
\end{equation*}
where $X^{t,\zeta;u,v}$ is introduced by equation (\ref{eq1})  and
\begin{eqnarray*}
 \Phi_{j}&=&\Phi_{j}(x):\mathbb{R}^n\rightarrow\mathbb{R},\\
f_{j}&=&f_{j}(t,x,y,z,u,v):[0,T]\times\mathbb{R}^n\times\mathbb{R}\times\mathbb{R}^d\times
U\times V\rightarrow \mathbb{R},
\end{eqnarray*}
are assumed to satisfy the conditions $(H3.3)$ and $(H3.4)$. In
addition, in order to simplify the arguments, we also suppose that

$(H3.5)$   $\Phi_{j}$ and $f_{j}$, $j=1,2,$  are bounded.

 The associated stochastic
backward semigroups are denoted by $^{j}G^{t,x;u,v}_{t,s}, t\leq s
\leq T, j=1,2,$ and for the associated cost functionals
$J_{j}(t,x;u,v)=\ ^{j}Y^{t,x;u,v}_t$, we have
\begin{eqnarray*}\label{}
J_{j}(t,x;u,v)&=&
\ ^{j}G^{t,x;u,v}_{t,T}[\Phi_{j}(X^{t,x;u,v}_T)]=\ ^{j}G^{t,x;u,v}_{t,t+\delta}[^{j}Y^{t,x;u,v}_{t+\delta}]\\
&=&^{j}G^{t,x;u,v}_{t,t+\delta}[J_{j}(t+\delta,X^{t,x;u,v}_{t+\delta},u,v)],
\end{eqnarray*}
$(t,x)\in[0,T]\times\mathbb{R}^{n},$ $(u,v)\in \mathcal
{U}_{t,T}\times \mathcal {V}_{t,T}, 0\leq \delta \leq T-t, j=1,2.$

For what follows we assume that the Isaacs condition holds in the
following sense: For all
$(t,x,y,p)\in[0,T]\times\mathbb{R}^{n}\times\mathbb{R}\times\mathbb{R}^{n}$
and $A\in \mathbb{S}^{n}$ (Recall that $\mathbb{S}^{n}$ denotes the
set of $n\times n$ symmetric matrices) and for $j=1,2,$ we have
\begin{eqnarray}\label{Isaacs}
\begin{aligned}
&\sup\limits_{u\in U}\inf\limits_{v\in V}
\Big\{\frac{1}{2}tr(\sigma\sigma^{T}(t,x,u,v)A)+\langle p,
b(t,x,u,v)\rangle+f_{j}(t,x,y,p^{T}\sigma(t,x,u,v),u,v) \Big\}\\
=&\inf\limits_{v\in V}\sup\limits_{u\in U}
\Big\{\frac{1}{2}tr(\sigma\sigma^{T}(t,x,u,v)A)+\langle p,
b(t,x,u,v)\rangle
  +f_{j}(t,x,y,p^{T}\sigma(t,x,u,v),u,v) \Big\}.
 \end{aligned}
\end{eqnarray}
Under the Isaacs condition  (\ref{Isaacs}) we have, similar to
\cite{BCQ},
\begin{eqnarray*}\label{}
W_{1}(t,x)&=& \esssup_{\alpha\in\mathcal {A}_{t,T}}
\essinf_{\beta\in\mathcal {B}_{t,T}}
J_{1}(t,x;\alpha,\beta)=\essinf_{\beta\in\mathcal
{B}_{t,T}}\esssup_{\alpha\in\mathcal {A}_{t,T}}
J_{1}(t,x;\alpha,\beta),
\end{eqnarray*}
and
\begin{eqnarray}\label{equa}
W_{2}(t,x)&=&\essinf_{\alpha\in\mathcal
{A}_{t,T}}\esssup_{\beta\in\mathcal {B}_{t,T}}
J_{2}(t,x;\alpha,\beta)=\esssup_{\beta\in\mathcal
{B}_{t,T}}\essinf_{\alpha\in\mathcal {A}_{t,T}}
J_{2}(t,x;\alpha,\beta),
\end{eqnarray}
$(t,x)\in[0,T]\times\mathbb{R}^{n}$.\\

Finally, we complete the preparation with  the definition of the
Nash equilibrium payoff of stochastic differential games, which is
similar to the definition introduced by Buckdahn, Cardaliaguet and
Rainer \cite{BCR}.
\begin{definition}\label{d2}
A couple $(e_{1},e_{2})\in\mathbb{R}^{2}$ is called a Nash
equilibrium payoff at the point $(t,x)$ if for any $\varepsilon>0$,
there exists $(\alpha_{\varepsilon},\beta_{\varepsilon})\in \mathcal
{A}_{t,T}\times \mathcal {B}_{t,T}$ such that, for all
$(\alpha,\beta)\in \mathcal {A}_{t,T}\times \mathcal {B}_{t,T},$
\begin{eqnarray}\label{eq4}
J_{1}(t,x;\alpha_{\varepsilon},\beta_{\varepsilon})\geq
J_{1}(t,x;\alpha,\beta_{\varepsilon})-\varepsilon,\
J_{2}(t,x;\alpha_{\varepsilon},\beta_{\varepsilon})\geq
J_{2}(t,x;\alpha_{\varepsilon},\beta)-\varepsilon,\
\mathbb{P}-a.s.,
\end{eqnarray}
and
\begin{eqnarray*}
|\mathbb{E}[J_{j}(t,x;\alpha_{\varepsilon},\beta_{\varepsilon})]-e_{j}|\leq
\varepsilon, \ j=1,2.
\end{eqnarray*}
\end{definition}

\begin{remark}\label{}
We attract the reader's attention to the fact that
$J_{j}(t,x;\alpha,\beta), j=1,2,$ are random variables. In our
existence result (Theorem \ref{t2}) we will construct cost
functionals $J_{j}(t,x;\alpha_{\varepsilon},\beta_{\varepsilon}),
\varepsilon>0, j=1,2,$ which are deterministic.
\end{remark}

By virtue of Lemma \ref{l2} we can easily get the following lemma.
\begin{lemma}\label{le3}
For any $\varepsilon>0$ and
$(\alpha_{\varepsilon},\beta_{\varepsilon})\in \mathcal
{A}_{t,T}\times \mathcal {B}_{t,T}$, (\ref{eq4}) holds if and only
if,  for all $(u,v)\in \mathcal {U}_{t,T}\times \mathcal {V}_{t,T}$,
\begin{eqnarray}\label{eq5}
J_{1}(t,x;\alpha_{\varepsilon},\beta_{\varepsilon})\geq
J_{1}(t,x;u,\beta_{\varepsilon}(u))-\varepsilon,\
J_{2}(t,x;\alpha_{\varepsilon},\beta_{\varepsilon})\geq
J_{2}(t,x;\alpha_{\varepsilon}(v),v)-\varepsilon, \ \mathbb{P}-a.s.
\end{eqnarray}
\end{lemma}

We now give the characterization theorem of Nash equilibrium payoffs
for nonzero-sum stochastic differential games.
\begin{theorem}\label{t1}
Let $(t,x)\in [0,T]\times \mathbb{R}^{n}$. Under Isaacs condition
(\ref{Isaacs}),  $(e_{1},e_{2})\in\mathbb{R}^{2}$ is a Nash
equilibrium payoff at point $(t,x)$ if and only if for all
$\varepsilon>0$, there exist
 $(u^{\varepsilon},v^{\varepsilon})\in \mathcal {U}_{t,T}\times \mathcal
 {V}_{t,T}$  such that for all $s\in[t,T]$ and $j=1,2,$\\
\begin{eqnarray}\label{eq6}
\mathbb{P}\Big(\ ^{j}Y^{t,x;u^{\varepsilon},v^{\varepsilon}}_{s}\geq
W_{j}(s,X^{t,x;u^{\varepsilon},v^{\varepsilon}}_{s})-\varepsilon\ |\
\mathcal {F}_{t}\Big)\geq 1-\varepsilon,\ \mathbb{P}-a.s.,
\end{eqnarray}
and
\begin{eqnarray}\label{eq7}
|\mathbb{E}[J_{j}(t,x;u^{\varepsilon},v^{\varepsilon})]-
e_{j}|\leq\varepsilon.
\end{eqnarray}
\end{theorem}

\begin{remark}
The above Theorem generalizes the results of \cite{BCR} and \cite{R}
from the case of classical cost functionals without running costs to
nonlinear cost functionals which running cost $f_{j}, j=1,2,$ depend
on $(y,z).$ Moreover, in our framework the controls can depend on
events occurring before time $t$.
\end{remark}

We prepare the proof of this Theorem by the following two lemmas.
\begin{lemma}\label{le1}
Let $(t,x)\in [0,T]\times \mathbb{R}^{n}$ and $u\in \mathcal
{U}_{t,T}$ be arbitrarily fixed. Then,

(i) for all $\delta\in[0,T-t]$ and $\varepsilon>0$,  there exists an
NAD strategy $\alpha\in \mathcal {A}_{t,T}$ such that, for all
 $v\in \mathcal {V}_{t,T}$,
\begin{eqnarray*}\label{}
\alpha(v)&=&u, \text{on}\ [t,t+\delta],\\
^{2}Y^{t,x;\alpha(v),v}_{t+\delta}&\leq&
W_{2}(t+\delta,X^{t,x;\alpha(v),v}_{t+\delta})+\varepsilon,\
\mathbb{P}-a.s.
\end{eqnarray*}

(ii) for all $\delta\in[0,T-t]$ and $\varepsilon>0$,  there exists an
NAD strategy  $\alpha\in \mathcal {A}_{t,T}$ such that, for all
 $v\in \mathcal {V}_{t,T}$,
\begin{eqnarray*}\label{}
\alpha(v)&=&u, \text{on}\ [t,t+\delta],\\
^{1}Y^{t,x;\alpha(v),v}_{t+\delta}&\geq&
W_{1}(t+\delta,X^{t,x;\alpha(v),v}_{t+\delta})-\varepsilon, \
\mathbb{P}-a.s.
\end{eqnarray*}
\end{lemma}

\begin{proof}
We only give the proof of (i). Indeed, (ii) can be proved by a
symmetric argument. We begin with observing that putting
$\beta^{v'}(u')=v'$, for all $u'\in \mathcal {U}_{t+\delta,T}$,
defines for every $v'\in \mathcal {V}_{t+\delta,T}$  an element
$\beta^{v'}\in \mathcal {B}_{t+\delta,T}$ and allows to regard
$\mathcal {V}_{t+\delta,T}$ as a subset of $\mathcal
{B}_{t+\delta,T}$. Consequently, from (\ref{equa}), for any $y\in
\mathbb{R}^{n}$, we have
\begin{eqnarray*}\label{}
W_{2}(t+\delta,y)&=&\essinf_{\alpha\in\mathcal
{A}_{t+\delta,T}}\esssup_{\beta\in\mathcal {B}_{t+\delta,T}}
J_{2}(t+\delta,y;\alpha,\beta)\\&\geq&\essinf_{\alpha\in\mathcal
{A}_{t+\delta,T}}\esssup_{v\in\mathcal {V}_{t+\delta,T}}
J_{2}(t+\delta,y;\alpha(v),v), \ \mathbb{P}-a.s.
\end{eqnarray*}
Therefore, for $\varepsilon_{0}>0$, there exists $\alpha_{y}\in
\mathcal {A}_{t+\delta,T}$ such that
\begin{eqnarray}\label{eqn1}
W_{2}(t+\delta,y)&\geq& \esssup_{v\in\mathcal
{V}_{t+\delta,T}}J_{2}(t+\delta,y; \alpha_{y}(v),v)-\varepsilon_{0},
\ \mathbb{P}-a.s.
\end{eqnarray}
(The existence of $\alpha_{y}\in \mathcal {A}_{t+\delta,T}$  can be
shown with the techniques used in the proof of Lemma 3.8 in Buckdahn
and Li \cite{BL2006}).

Let $\{O_{i}\}_{i\geq1}\subset\mathcal {B}(\mathbb{R}^{n})$  be a
partition of $\mathbb{R}^{n}$ such that
$\sum\limits_{i\geq1}O_{i}=\mathbb{R}^{n}, O_{i}\neq\emptyset,$ and
$\text{diam}(O_{i})\leq\varepsilon_{0},i\geq1$. Let $y_{i}\in O_{i},
i\geq 1$. We put, for $v\in\mathcal {V}_{t,T}$,
\begin{equation}\label{equation}
\alpha(v)_{s}=\left\{
\begin{array}{rl}
 u_s, \qquad \qquad\qquad\qquad s\in [t, t+\delta]
,\\
\sum\limits_{i\geq1}1_{\{X^{t,x;u,v}_{t+\delta}\in
O_{i}\}}\alpha_{y_{i}}(v|_{[t+\delta,T]})_{s}, s\in (t+\delta,T].
\end{array}
\right.
\end{equation}
The such introduced mapping  $\alpha:\mathcal {V}_{t,T}\rightarrow
\mathcal {U}_{t,T}$ is an NAD strategy. Indeed,\\

 $(i)$ {\it The mapping $\alpha$ is
nonanticipative}. Proof: For every $(\mathcal {F}_{r})$-stopping
time $\tau:\Omega\rightarrow [t,T],$ and for
 $v_{1},v_{2} \in\mathcal {V}_{t,T}$ with $v_{1}=v_{2}$   on $[[ t,\tau]]$,
 we decompose $v_{1},v_{2}$ into $v_{1}^{1},v_{2}^{1}\in\mathcal {V}_{t,t+\delta}$
 and $v_{1}^{2},v_{2}^{2}\in\mathcal {V}_{t+\delta,T}$ such that
$v_{i}^{1}= v_{i}$ on $[t,t+\delta]$, and $v_{i}^{2}= v_{i}$ on
$[t+\delta,T], i=1,2.$ In order to abbreviate, we will write for
this:
 $v_{1}=v_{1}^{1}\oplus v_{1}^{2}$  and $v_{2}=v_{2}^{1}\oplus v_{2}^{2}$.  Then we have
 $v_{1}^{1}= v_{2}^{1}$  on $[[ t,\tau\wedge(t+\delta)]]$ and
 $v_{1}^{2}= v_{2}^{2}$ on $[[ \tau\wedge(t+\delta), \tau]]$. It is obvious
 that  $\alpha(v_{1})=u=\alpha(v_{2})$ on $[t,t+\delta]$ and, hence, also on $[[ t,\tau\wedge(t+\delta)]]$.
Since $v_{1}^{1}= v_{2}^{1}$ on $[[ t,\tau\wedge(t+\delta)]]$, we
have
$X^{t,x;\alpha(v_{1}^{1}),v_{1}^{1}}_{t+\delta}=X^{t,x;\alpha(v^{1}_{2}),v^{1}_{2}}_{t+\delta}$
on $\{\tau>t+\delta\}, \mathbb{P}-a.s.$ Therefore, because of the
nonanticipativity  of $\alpha_{y_{i}}, i\geq 1,$
$$\alpha(v_{1})=\sum\limits_{i\geq1}1_{\{X^{t,x;\alpha(v_{1}^{1}),v_{1}^{1}}_{t+\delta}\in
O_{i}\}}
\alpha_{y_{i}}(v_{1}^{2})=\sum\limits_{i\geq1}1_{\{X^{t,x;\alpha(v_{2}^{1}),v_{2}^{1}}_{t+\delta}\in
O_{i}\}} \alpha_{y_{i}}(v_{2}^{2})=\alpha(v_{2})$$ on $[[
\tau\wedge(t+\delta), \tau]]$.\\

 $(ii)$  { \it The mapping  $\alpha$ is a nonanticipative strategy with
delay.}  Proof: For $v=v^{1}\oplus v^{2}\in\mathcal
{V}_{t,t+\delta}\times\mathcal {V}_{t+\delta,T}$, we have
\begin{eqnarray*}\label{}
\alpha(v)=u\oplus\sum\limits_{i\geq1}1_{\{X^{t,x;u,v^{1}}_{t+\delta}\in
O_{i}\}}\alpha_{y_{i}}(v^{2}).
\end{eqnarray*}
Let $\{S_{n}^{i}(v^{2})\}_{n\geq 1}$ be the sequence of the stopping
times
 associated with $\alpha_{y_{i}}\in \mathcal {A}_{t+\delta,T}$ in the sense of
 the Definition \ref{d1}.  Then, putting  $S_{0}=t, S_{1}=t+\delta$,
\begin{eqnarray*}\label{}
S_{n+1}(v)=\sum\limits_{i\geq1}1_{\{X^{t,x;u,v^{1}}_{t+\delta}\in
O_{i}\}} S_{n}^{i}(v^{2}), n\geq1.
\end{eqnarray*}
We have that  $\{S_{n}(v)\}_{n\geq 1}$  satisfies the condition 2)
in Definition \ref{d1}. Thus, $\alpha$ is  a nonanticipative
strategy with delay.

From Lemma \ref{le5}, (\ref{eqn1}) and (\ref{equation}) it follows
that, for $v\in\mathcal {V}_{t,T}$,
\begin{eqnarray*}\label{}
&&W_{2}(t+\delta,X^{t,x;\alpha(v),v}_{t+\delta})\\
&\geq&
\sum\limits_{i\geq1}1_{\{X^{t,x;\alpha(v),v}_{t+\delta}\in
O_{i}\}}W_{2}(t+\delta,y_{i})-C\varepsilon_{0}\\
&\geq& \sum\limits_{i\geq1}1_{\{X^{t,x;\alpha(v),v}_{t+\delta}\in
O_{i}\}} J_{2}(t+\delta,y_{i};
\alpha_{y_{i}}(v|_{[t+\delta,T]}),v)-C\varepsilon_{0}\\
&=& \sum\limits_{i\geq1}1_{\{X^{t,x;\alpha(v),v}_{t+\delta}\in
O_{i}\}} J_{2}(t+\delta,y_{i};\alpha(v),v)-C\varepsilon_{0}.
\end{eqnarray*}
Thus, from Proposition \ref{p3},
\begin{eqnarray*}\label{}
W_{2}(t+\delta,X^{t,x;\alpha(v),v}_{t+\delta}) &\geq&
\sum\limits_{i\geq1}1_{\{X^{t,x;\alpha(v),v}_{t+\delta}\in O_{i}\}}
J_{2}(t+\delta,X^{t,x;\alpha(v),v}_{t+\delta};
\alpha(v),v)-C\varepsilon_{0}\\
&=&  J_{2}(t+\delta,X^{t,x;\alpha(v),v}_{t+\delta};
\alpha(v),v)-C\varepsilon_{0}.
\end{eqnarray*}
Here $C$ is a constant which can vary from line to line, but which
is independent of $v\in\mathcal {V}_{t,T}$.

 Putting
$\varepsilon_{0}=\varepsilon C^{-1}$ in the latter estimate, we
obtain
\begin{eqnarray*}\label{}
W_{2}(t+\delta,X^{t,x;\alpha(v),v}_{t+\delta})&\geq&
 J_{2}(t+\delta,X^{t,x;\alpha(v),v}_{t+\delta};
\alpha(v),v)-\varepsilon, v\in\mathcal {V}_{t,T}.
\end{eqnarray*}
The proof is complete.
\end{proof}\\

The proof of Theorem \ref{t1} necessitates the following lemma.

\begin{lemma}\label{le2}
There exists a positive constant  $C$ such that, for all
$(u,v),(u',v')\in \mathcal {U}_{t,T}\times \mathcal {V}_{t,T}$, and
for all $\mathcal {F}_{r}$-stopping times $S:\Omega\rightarrow[t,T]$
with $X^{t,x;u,v}_{S}=X^{t,x;u',v'}_{S}$, $\mathbb{P}-a.s.,$ it
holds,  for all real $\tau\in[t,T],$

\begin{eqnarray*}\label{}
\mathbb{E}[\sup\limits_{0\leq s\leq \tau}|X^{t,x;u,v}_{(S+s)\wedge
T}-X^{t,x;u',v'}_{(S+s)\wedge T}|^{2} \Big|\mathcal {F}_{t}]\leq
C\tau, \ \mathbb{P}-a.s.
\end{eqnarray*}
\end{lemma}
This lemma is the result of a straight forward estimate using the
fact that $b$ and $\sigma$ are bounded.\\

Let us  give now the proof of Theorem \ref{t1}.\\

{\bf Proof of Theorem \ref{t1}: \underline {Sufficiency}} of
(\ref{eq6}) and (\ref{eq7}).\\

\begin{proof}
Let  $\varepsilon>0$ be arbitrarily fixed. For $\varepsilon_{0}>0$
being specified later we suppose that
$(u^{\varepsilon_{0}},v^{\varepsilon_{0}})\in\mathcal
{U}_{t,T}\times \mathcal {V}_{t,T}$ satisfies (\ref{eq6}) and
(\ref{eq7}), i.e., for all $s\in[t,T]$ and $j=1,2,$
\begin{eqnarray}\label{e12}
\mathbb{P}\Big(\
^{j}Y^{t,x;u^{\varepsilon_{0}},v^{\varepsilon_{0}}}_{s}\geq
W_{j}(s,X^{t,x;u^{\varepsilon_{0}},v^{\varepsilon_{0}}}_{s})-\varepsilon_{0}\
|\ \mathcal {F}_{t}\Big)\geq 1-\varepsilon_{0}, \ \mathbb{P}-a.s.,
\end{eqnarray}
and
\begin{eqnarray}\label{e13}
|\mathbb{E}[J_{j}(t,x;u^{\varepsilon_{0}},v^{\varepsilon_{0}})]-
e_{j}|\leq\varepsilon_{0}.
\end{eqnarray}
Let us  fix some partition: $t=t_{0}\leq t_{1}\leq\cdots \leq
t_{m}=T$ of $[t,T]$ and  $\tau=\sup\limits_{i}|t_{i}-t_{i+1}|$.
 We apply  Lemma \ref{le1} to $u^{\varepsilon_{0}}$ and
$t+\delta=t_{1},\cdots,t_{m}$, successively. Then, for
$\varepsilon_{1}>0$ ($\varepsilon_{1}$ depends on $\varepsilon$ and
is specified later) there exist NAD strategies $\alpha_{i}\in
\mathcal {A}_{t,T}, i=1,\cdots,m$,  such that,  for all
 $v\in \mathcal {V}_{t,T}$,
\begin{eqnarray}\label{eq8}
\alpha_{i}(v)&=&u^{\varepsilon_{0}}, \text{on}\ [t,t_{i}],\nonumber\\
^{2}Y^{t,x;\alpha_{i}(v),v}_{t_{i}}&\leq&
W_{2}(t_{i},X^{t,x;\alpha_{i}(v),v}_{t_{i}})+\varepsilon_{1},
\mathbb{P}-a.s.
\end{eqnarray}
For all $v\in \mathcal {V}_{t,T}$, we define
\begin{eqnarray*}\label{}
S^{v}&=&\inf\Big\{s\geq t \ |\ \lambda(\{r\in[t,s]: v_{r}\neq
v_{r}^{\varepsilon_{0}}\})>0\Big\},\\
t^{v}&=&\inf\Big\{t_{i}\geq S^{v}\  |\ i=1,\cdots, m\Big\}\wedge T,
\end{eqnarray*}
where $\lambda$ denotes the Lebesgue measure on the real line
$\mathbb{R}$. Obviously, $S^{v}$ and $t^{v}$ are stopping times, and
we have $S^{v}\leq t^{v}\leq S^{v}+\tau$.

 Let
\begin{equation*}\label{}
\alpha_{\varepsilon}(v)=\left\{
\begin{array}{ll}
 u^{\varepsilon_{0}}, & \ \text{on}\ [[ t, t^{v}]],\\
\alpha_{i}(v), & \ \text{on}\ (t_{i},T]\times \{t^{v}=t_{i}\}, 1\leq
i\leq m.
\end{array}
\right.
\end{equation*}
It is easy to check that $\alpha_{\varepsilon}$ is an NAD strategy.
From (\ref{eq8}) it follows that
\begin{eqnarray}\label{eq9}
^{2}Y^{t,x;\alpha_{\varepsilon}(v),v}_{t^{v}}&=&
\sum\limits_{i=1}^{m} \
^{2}Y^{t,x;\alpha_{\varepsilon}(v),v}_{t_{i}}1_{\{t^{v}=t_{i}\}}\nonumber\\
&\leq&
\sum\limits_{i=1}^{m}W_{2}(t_{i},X^{t,x;\alpha_{\varepsilon}(v),v}_{t_{i}})1_{\{t^{v}=t_{i}\}}
+\varepsilon_{1}\nonumber\\
&=&
W_{2}(t^{v},X^{t,x;\alpha_{\varepsilon}(v),v}_{t^{v}})+\varepsilon_{1},
\ \mathbb{P}-a.s.
\end{eqnarray}
In what follows we will show that, for all $\varepsilon>0$ and $v\in
\mathcal {V}_{t,T}$,
\begin{eqnarray}\label{e1}
 J_{2}(t,x;\alpha_{\varepsilon}(v),v)
\leq J_{2}(t,x;u^{\varepsilon_{0}},v^{\varepsilon_{0}})
+\varepsilon, \quad
\alpha_{\varepsilon}(v^{\varepsilon_{0}})=u^{\varepsilon_{0}}.
\end{eqnarray}
This relation as well as the symmetric one for $J_{1}$ will lead to
the sufficiency of (\ref{eq6}) and (\ref{eq7}).  For the proof of
(\ref{e1}), we  note that by (\ref{eq9}), Lemma \ref{l1} and Lemma
\ref{l3} there exists a positive constant $C$ such that
\begin{eqnarray}\label{eqn2}
 J_{2}(t,x,\alpha_{\varepsilon}(v),v)&=&
 ^{2}G^{t,x;\alpha_{\varepsilon}(v),v}_{t,t^{v}}[^{2}Y^{t,x,\alpha_{\varepsilon}(v),v}_{t^{v}}]\nonumber\\
&\leq &
^{2}G^{t,x;\alpha_{\varepsilon}(v),v}_{t,t^{v}}[W_{2}(t^{v},X^{t,x;\alpha_{\varepsilon}(v),v}_{t^{v}})
+\varepsilon_{1}]\nonumber\\
&\leq &
^{2}G^{t,x;\alpha_{\varepsilon}(v),v}_{t,t^{v}}[W_{2}(t^{v},X^{t,x;\alpha_{\varepsilon}(v),v}_{t^{v}})]
+C\varepsilon_{1}.
\end{eqnarray}
Thanks to the Lemmas \ref{le5} and \ref{le2} as well as the
definitions of $t^{v}$ and $\alpha_{\varepsilon}$ we have
\begin{eqnarray*}\label{}
&&\mathbb{E}[|W_{2}(t^{v},X_{t^{v}}^{t,x;u^{\varepsilon_{0}},v^{\varepsilon_{0}}})
-W_{2}(t^{v},X^{t,x;\alpha_{\varepsilon}(v),v}_{t^{v}})|^{2}
\Big|\mathcal {F}_{t}]\\
 &\leq & C \mathbb{E}[|X_{t^{v}}^{t,x;u^{\varepsilon_{0}},v^{\varepsilon_{0}}}
-X^{t,x;\alpha_{\varepsilon}(v),v}_{t^{v}}|^{2}
\Big|\mathcal {F}_{t}]\\
 &\leq & C\tau,\  \mathbb{P}-a.s.
\end{eqnarray*}
Thus, from Lemma \ref{l3} it follows that
\begin{eqnarray*}\label{}
\begin{aligned}
&|\
^{2}G^{t,x;\alpha_{\varepsilon}(v),v}_{t,t^{v}}[W_{2}(t^{v},X_{t^{v}}^{t,x;u^{\varepsilon_{0}},v^{\varepsilon_{0}}})]
-\ ^{2}G^{t,x;\alpha_{\varepsilon}(v),v}_{t,t^{v}}[W_{2}(t^{v},X^{t,x;\alpha_{\varepsilon}(v),v}_{t^{v}})]|\\
&\leq
C\mathbb{E}[|W_{2}(t^{v},X_{t^{v}}^{t,x;u^{\varepsilon_{0}},v^{\varepsilon_{0}}})
-W_{2}(t^{v},X^{t,x;\alpha_{\varepsilon}(v),v}_{t^{v}})|^{2}
\Big|\mathcal {F}_{t}]^{\frac{1}{2}}\\
&\leq  C\tau^{\frac{1}{2}},
\end{aligned}
\end{eqnarray*}
and the above inequality and (\ref{eqn2}) yield
\begin{eqnarray*}\label{}
&& J_{2}(t,x,\alpha_{\varepsilon}(v),v)\\
&\leq &
^{2}G^{t,x;\alpha_{\varepsilon}(v),v}_{t,t^{v}}[W_{2}(t^{v},X_{t^{v}}^{t,x;u^{\varepsilon_{0}},v^{\varepsilon_{0}}})]
- \ ^{2}G^{t,x;\alpha_{\varepsilon}(v),v}_{t,t^{v}}[W_{2}(t^{v},X^{t,x;\alpha_{\varepsilon}(v),v}_{t^{v}})]|\\
&&+|\ ^{2}G^{t,x;\alpha_{\varepsilon}(v),v}_{t,t^{v}}[W_{2}(t^{v},X_{t^{v}}^{t,x;u^{\varepsilon_{0}},v^{\varepsilon_{0}}})]
+C\varepsilon_{1}\\
&\leq &
^{2}G^{t,x;\alpha_{\varepsilon}(v),v}_{t,t^{v}}[W_{2}(t^{v},X_{t^{v}}^{t,x;u^{\varepsilon_{0}},v^{\varepsilon_{0}}})]
+C\varepsilon_{1}+C\tau^{\frac{1}{2}}.
\end{eqnarray*}
For $s\in[t,T]$, we put
\begin{eqnarray*}\label{}
\Omega_{s}=\Big\{\
^{2}Y^{t,x;u^{\varepsilon_{0}},v^{\varepsilon_{0}}}_{s}\geq
W_{2}(s,X^{t,x;u^{\varepsilon_{0}},v^{\varepsilon_{0}}}_{s})-\varepsilon_{0}\Big\}.
\end{eqnarray*}
By the inequality $a\leq b+|a-b|, a,b\in\mathbb{R}$, we have
\begin{eqnarray}\label{eqn4}
&& J_{2}(t,x;\alpha_{\varepsilon}(v),v)\nonumber\\
 &\leq& \ ^{2}G^{t,x;\alpha_{\varepsilon}(v),v}_{t,t^{v}}
 [W_{2}(t^{v},X_{t^{v}}^{t,x;u^{\varepsilon_{0}},v^{\varepsilon_{0}}})]
+C\varepsilon_{1}+C\tau^{\frac{1}{2}}\nonumber\\
 &=& \ ^{2}G^{t,x;\alpha_{\varepsilon}(v),v}_{t,t^{v}}
[\sum\limits_{i=1}^{m}W_{2}(t_{i},X_{t_{i}}^{t,x;u^{\varepsilon_{0}},v^{\varepsilon_{0}}})1_{\{t^{v}=t_{i}\}}]
+C\varepsilon_{1}+C\tau^{\frac{1}{2}}\nonumber\\
&\leq&^{2}G^{t,x;\alpha_{\varepsilon}(v),v}_{t,t^{v}}
[\sum\limits_{i=1}^{m}W_{2}(t_{i},X_{t_{i}}^{t,x;u^{\varepsilon_{0}},v^{\varepsilon_{0}}})
1_{\{t^{v}=t_{i}\}}1_{\Omega_{t_{i}}}]+C\varepsilon_{1}+C\tau^{\frac{1}{2}}
\nonumber\\&&+|\ ^{2}G^{t,x;\alpha_{\varepsilon}(v),v}_{t,t^{v}}
[\sum\limits_{i=1}^{m}W_{2}(t_{i},X_{t_{i}}^{t,x;u^{\varepsilon_{0}},v^{\varepsilon_{0}}})1_{\{t^{v}=t_{i}\}}]
\nonumber\\&& \quad -\
^{2}G^{t,x;\alpha_{\varepsilon}(v),v}_{t,t^{v}}
[\sum\limits_{i=1}^{m}W_{2}(t_{i},X_{t_{i}}^{t,x;u^{\varepsilon_{0}},v^{\varepsilon_{0}}})
1_{\{t^{v}=t_{i}\}}1_{\Omega_{t_{i}}}]|.
\end{eqnarray}
Using Lemma \ref{l3} again as well as the boundedness of $W_{2}$, we
see that
\begin{eqnarray}\label{eqn5}
&&|\ ^{2}G^{t,x;\alpha_{\varepsilon}(v),v}_{t,t^{v}}
[\sum\limits_{i=1}^{m}W_{2}(t_{i},X_{t_{i}}^{t,x;u^{\varepsilon_{0}},v^{\varepsilon_{0}}})1_{\{t^{v}=t_{i}\}}]
\nonumber\\&&\quad -\
^{2}G^{t,x;\alpha_{\varepsilon}(v),v}_{t,t^{v}}
[\sum\limits_{i=1}^{m}W_{2}(t_{i},X_{t_{i}}^{t,x;u^{\varepsilon_{0}},v^{\varepsilon_{0}}})
1_{\{t^{v}=t_{i}\}}1_{\Omega_{t_{i}}}]|\nonumber\\
&\leq& C
\mathbb{E}[\sum\limits_{i=1}^{m}|W_{2}(t_{i},X_{t_{i}}^{t,x;u^{\varepsilon_{0}},v^{\varepsilon_{0}}})|^{2}
1_{\{t^{v}=t_{i}\}}1_{\Omega^{c}_{t_{i}}} \Big|\mathcal
{F}_{t}]^{\frac{1}{2}}\nonumber\\
&\leq& C\sum\limits_{i=1}^{m} \mathbb{P}(\Omega^{c}_{t_{i}}
|\mathcal{F}_{t})^{\frac{1}{2}}\leq Cm\varepsilon_{0}^{\frac{1}{2}},
\end{eqnarray}
where we have used (\ref{e12}) for the latter estimate.
 Observing that
\begin{eqnarray*}\label{}
^{2}Y^{t,x;u^{\varepsilon_{0}},v^{\varepsilon_{0}}}_{t_{i}}\geq
W_{2}(t_{i},X^{t,x;u^{\varepsilon_{0}},v^{\varepsilon_{0}}}_{t_{i}})-\varepsilon_{0},
\ \text{on}\ \Omega_{t_{i}},
\end{eqnarray*}
we deduce  from the Lemmas \ref{l1} and \ref{l3} that
\begin{eqnarray*}\label{}
&&^{2}G^{t,x;\alpha_{\varepsilon}(v),v}_{t,t^{v}}
[\sum\limits_{i=1}^{m}W_{2}(t_{i},X_{t_{i}}^{t,x;u^{\varepsilon_{0}},v^{\varepsilon_{0}}})
1_{\{t^{v}=t_{i}\}}1_{\Omega_{t_{i}}}]\\
&\leq& ^{2}G^{t,x;\alpha_{\varepsilon}(v),v}_{t,t^{v}}
[\sum\limits_{i=1}^{m}(^{2}Y^{t,x;u^{\varepsilon_{0}},v^{\varepsilon_{0}}}_{t_{i}}+\varepsilon_{0})
1_{\{t^{v}=t_{i}\}}1_{\Omega_{t_{i}}}]\\
&\leq& ^{2}G^{t,x;\alpha_{\varepsilon}(v),v}_{t,t^{v}}
[\sum\limits_{i=1}^{m}\
^{2}Y^{t,x;u^{\varepsilon_{0}},v^{\varepsilon_{0}}}_{t_{i}}
1_{\{t^{v}=t_{i}\}}1_{\Omega_{t_{i}}}+\varepsilon_{0}]\\
&\leq& ^{2}G^{t,x;\alpha_{\varepsilon}(v),v}_{t,t^{v}}
[\sum\limits_{i=1}^{m}\
^{2}Y^{t,x;u^{\varepsilon_{0}},v^{\varepsilon_{0}}}_{t_{i}}
1_{\{t^{v}=t_{i}\}}1_{\Omega_{t_{i}}}]+C\varepsilon_{0}.
\end{eqnarray*}
Hence, taking into account
$^{2}Y^{t,x;\alpha_{\varepsilon}(v),v}_{t^{v}}=
\sum\limits_{i=1}^{m} \
^{2}Y^{t,x;\alpha_{\varepsilon}(v),v}_{t_{i}}1_{\{t^{v}=t_{i}\}}$
and that, in analogy to (\ref{eqn5})
\begin{eqnarray*}\label{}
&&|\ ^{2}G^{t,x;\alpha_{\varepsilon}(v),v}_{t,t^{v}}
[\sum\limits_{i=1}^{m}\
^{2}Y^{t,x;u^{\varepsilon_{0}},v^{\varepsilon_{0}}}_{t_{i}}
1_{\{t^{v}=t_{i}\}}1_{\Omega_{t_{i}}}] -\
^{2}G^{t,x;\alpha_{\varepsilon}(v),v}_{t,t^{v}}
[\sum\limits_{i=1}^{m} \
^{2}Y^{t,x;\alpha_{\varepsilon}(v),v}_{t_{i}}1_{\{t^{v}=t_{i}\}}]|\\
&&\leq Cm\varepsilon_{0}^{\frac{1}{2}},
\end{eqnarray*}
we see that
\begin{eqnarray*}\label{}
&&^{2}G^{t,x;\alpha_{\varepsilon}(v),v}_{t,t^{v}}
[\sum\limits_{i=1}^{m}W_{2}(t_{i},X_{t_{i}}^{t,x;u^{\varepsilon_{0}},v^{\varepsilon_{0}}})
1_{\{t^{v}=t_{i}\}}1_{\Omega_{t_{i}}}]\\
&\leq& ^{2}G^{t,x;\alpha_{\varepsilon}(v),v}_{t,t^{v}} [
^{2}Y^{t,x;u^{\varepsilon_{0}},v^{\varepsilon_{0}}}_{t^{v}}]+C\varepsilon_{0}+ Cm\varepsilon_{0}^{\frac{1}{2}}\\
&\leq&|\ ^{2}G^{t,x;\alpha_{\varepsilon}(v),v}_{t,t^{v}} [ \
^{2}Y^{t,x;u^{\varepsilon_{0}},v^{\varepsilon_{0}}}_{t^{v}}] -\
^{2}G^{t,x;u^{\varepsilon_{0}},v^{\varepsilon_{0}}}_{t,t^{v}} [ \
^{2}Y^{t,x;u^{\varepsilon_{0}},v^{\varepsilon_{0}}}_{t^{v}}]|\\&& +
\ ^{2}G^{t,x;u^{\varepsilon_{0}},v^{\varepsilon_{0}}}_{t,t^{v}} [ \
^{2}Y^{t,x;u^{\varepsilon_{0}},v^{\varepsilon_{0}}}_{t^{v}}]
+C\varepsilon_{0}+ Cm\varepsilon_{0}^{\frac{1}{2}}\\
&=&|\ ^{2}G^{t,x;\alpha_{\varepsilon}(v),v}_{t,t^{v}} [ \
^{2}Y^{t,x;u^{\varepsilon_{0}},v^{\varepsilon_{0}}}_{t^{v}}] -\
^{2}G^{t,x;u^{\varepsilon_{0}},v^{\varepsilon_{0}}}_{t,t^{v}} [
\ ^{2}Y^{t,x;u^{\varepsilon_{0}},v^{\varepsilon_{0}}}_{t^{v}}]|\\
&&+J_{2}(t,x;u^{\varepsilon_{0}},v^{\varepsilon_{0}})
+C\varepsilon_{0}+ Cm\varepsilon_{0}^{\frac{1}{2}}.
\end{eqnarray*}
In the frame of the proof  of (\ref{e1}) we also need the following
estimate
\begin{eqnarray*}\label{}
|\ ^{2}G^{t,x;\alpha_{\varepsilon}(v),v}_{t,t^{v}} [ \
^{2}Y^{t,x;u^{\varepsilon_{0}},v^{\varepsilon_{0}}}_{t^{v}}] -\
^{2}G^{t,x;u^{\varepsilon_{0}},v^{\varepsilon_{0}}}_{t,t^{v}} [ \
^{2}Y^{t,x;u^{\varepsilon_{0}},v^{\varepsilon_{0}}}_{t^{v}}]|\leq
C\tau^{\frac{1}{2}}.
\end{eqnarray*}
In order to verify this relation we let, for all $s\in [t,t^{v}]$,
 $$y_{s}=\ ^{2}G^{t,x;\alpha_{\varepsilon}(v),v}_{s,t^{v}} [\
^{2}Y^{t,x;u^{\varepsilon_{0}},v^{\varepsilon_{0}}}_{t^{v}}],$$ and
we consider the  BSDE solved  by $y=(y_{s})$
\begin{equation*}\label{}
\begin{array}{lll}
y_s &=& ^{2}Y^{t,x;u^{\varepsilon_{0}},v^{\varepsilon_{0}}}_{t^{v}}
+\displaystyle \int_s^{t^{v}}
f_{2}(r,X^{t,x;\alpha_{\varepsilon}(v),v}_r,
y_r,z_r,\alpha_{\varepsilon}(v_r), v_r)dr-\displaystyle
\int_s^{t^{v}} z_rdB_r,
\end{array}
\end{equation*}
as well as
\begin{equation*}\label{}
\begin{array}{lll}
^{2}Y^{t,x;u^{\varepsilon_{0}},v^{\varepsilon_{0}}}_s &=&
^{2}Y^{t,x;u^{\varepsilon_{0}},v^{\varepsilon_{0}}}_{t^{v}}
+\displaystyle \int_s^{t^{v}}
f_{2}(r,X^{t,x;u^{\varepsilon_{0}},v^{\varepsilon_{0}}}_r, \
^{2}Y^{t,x;u^{\varepsilon_{0}},v^{\varepsilon_{0}}}_r,
\ ^{2}Z^{t,x;u^{\varepsilon_{0}},v^{\varepsilon_{0}}}_r,u^{\varepsilon_{0}}_r, v^{\varepsilon_{0}}_r)dr\\
&&-\displaystyle \int_s^{t^{v}} \
^{2}Z^{t,x;u^{\varepsilon_{0}},v^{\varepsilon_{0}}}_rdB_r, s\in [t,
t^{v}].
\end{array}
\end{equation*}
We notice that  $\alpha_{\varepsilon}(v)=u^{\varepsilon_{0}}, $ on
$[[ t, t^{v}]]$, $v=v^{\varepsilon_{0}}, $ on $[[ t, S^{v}]]$. (Of
course, these equalities, in particular, the latter one, are
understood as $dsd\mathbb{P}-a.e.$)
 Thanks to Lemma \ref{l3} we have
\begin{eqnarray*}
&&|\ ^{2}G^{t,x;\alpha_{\varepsilon}(v),v}_{t,t^{v}} [
^{2}Y^{t,x;u^{\varepsilon_{0}},v^{\varepsilon_{0}}}_{t^{v}}] - \
^{2}G^{t,x;u^{\varepsilon_{0}},v^{\varepsilon_{0}}}_{t,t^{v}} [
^{2}Y^{t,x;u^{\varepsilon_{0}},v^{\varepsilon_{0}}}_{t^{v}}]|^{2}\\
&\leq & C \mathbb{E}[\int_t^{S^{v}}
|f_{2}(r,X^{t,x;\alpha_{\varepsilon}(v),v}_r,y_r,z_r,\alpha_{\varepsilon}(v)_r,
v_r)-f_{2}(r,X^{t,x;u^{\varepsilon_{0}},v^{\varepsilon_{0}}}_r, y_r,
z_r,u^{\varepsilon_{0}}_r, v^{\varepsilon_{0}}_r) |^2dr|\mathcal
{F}_t]\\
&&+C \mathbb{E}[\int_{S^{v}}^{t^{v}}
|f_{2}(r,X^{t,x;\alpha_{\varepsilon}(v),v}_r,y_r,z_r,\alpha_{\varepsilon}(v)_r,
v_r)-f_{2}(r,X^{t,x;u^{\varepsilon_{0}},v^{\varepsilon_{0}}}_r, y_r,
z_r,u^{\varepsilon_{0}}_r, v^{\varepsilon_{0}}_r)
|^2dr|\mathcal{F}_t]\\
&=&C\mathbb{E}[\int_{S^{v}}^{t^{v}}
|f_{2}(r,X^{t,x;\alpha_{\varepsilon}(v),v}_r,y_r,z_r,\alpha_{\varepsilon}(v)_r,
v_r)-f_{2}(r,X^{t,x;u^{\varepsilon_{0}},v^{\varepsilon_{0}}}_r, y_r,
z_r,u^{\varepsilon_{0}}_r, v^{\varepsilon_{0}}_r)
|^2dr|\mathcal{F}_t]\\
 &\leq & C \mathbb{E}[\int_{S^{v}}^{t^{v}} 1_{\{v_r\neq
v^{\varepsilon_{0}}_r\}} dr|\mathcal {F}_t]\leq  C
\mathbb{E}[t^{v}-S^{v}|\mathcal {F}_t]\leq C\tau.
\end{eqnarray*}
Therefore, we have
\begin{eqnarray*}\label{}
&&^{2}G^{t,x;\alpha_{\varepsilon}(v),v}_{t,t^{v}}
[\sum\limits_{i=1}^{m}W_{2}(t_{i},X_{t_{i}}^{t,x;u^{\varepsilon_{0}},v^{\varepsilon_{0}}})
1_{\{t^{v}=t_{i}\}}1_{\Omega_{t_{i}}}]\\
&\leq&C\tau^{\frac{1}{2}}
+J_{2}(t,x;u^{\varepsilon_{0}},v^{\varepsilon_{0}})
+C\varepsilon_{0}+Cm\varepsilon_{0}^{\frac{1}{2}},
\end{eqnarray*}
and (\ref{eqn4}), (\ref{eqn5}) as well as  the above inequality
yield
\begin{eqnarray*}\label{}
 J_{2}(t,x;\alpha_{\varepsilon}(v),v)
&\leq&J_{2}(t,x;u^{\varepsilon_{0}},v^{\varepsilon_{0}})
+C\varepsilon_{0}+Cm\varepsilon_{0}^{\frac{1}{2}}
+C\varepsilon_{1}+C\tau^{\frac{1}{2}}.
\end{eqnarray*}
We can choose $\tau>0, \varepsilon_{0}>0,$ and $ \varepsilon_{1}>0$
such that $C\varepsilon_{0}+Cm\varepsilon_{0}^{\frac{1}{2}}
+C\varepsilon_{1}+C\tau^{\frac{1}{2}}\leq \varepsilon$ and
$\varepsilon_{0}<\varepsilon.$ Thus,
\begin{eqnarray*}\label{}
 J_{2}(t,x;\alpha_{\varepsilon}(v),v)
\leq J_{2}(t,x;u^{\varepsilon_{0}},v^{\varepsilon_{0}})
+\varepsilon, v\in \mathcal {V}_{t,T}.
\end{eqnarray*}
By a symmetric argument we can construct $\beta_{\varepsilon}\in
\mathcal {B}_{t,T}$ such that, for all $u\in \mathcal {U}_{t,T}$,
\begin{eqnarray}\label{e2}
 J_{1}(t,x;u,\beta_{\varepsilon}(u))
\leq J_{1}(t,x;u^{\varepsilon_{0}},v^{\varepsilon_{0}})
+\varepsilon,\quad
\beta_{\varepsilon}(u^{\varepsilon_{0}})=v^{\varepsilon_{0}}.
\end{eqnarray}
Finally, by virtue of  (\ref{e1}), (\ref{e2}), (\ref{e13}) and Lemma
\ref{le3} we can see that
$(\alpha_{\varepsilon},\beta_{\varepsilon})$ satisfies Definition
\ref{d2}. Therefore, $(e_{1},e_{2})$ is a Nash equilibrium payoff.
\end{proof}\\

{\bf Proof of Theorem \ref{t1}: \underline {Necessity}} of
(\ref{eq6}) and (\ref{eq7}).\\

\begin{proof}
We assume that $(e_{1},e_{2})\in\mathbb{R}^{2}$ is a Nash
equilibrium payoff at the point $(t,x)$. Then, for all sufficiently small
$\varepsilon>0$, there exists
$(\alpha_{\varepsilon},\beta_{\varepsilon})\in \mathcal
{A}_{t,T}\times \mathcal {B}_{t,T}$ such that, for all
$(\alpha,\beta)\in \mathcal {A}_{t,T}\times \mathcal {B}_{t,T}$
\begin{eqnarray}\label{e9}
J_{1}(t,x;\alpha_{\varepsilon},\beta_{\varepsilon})\geq
J_{1}(t,x;\alpha,\beta_{\varepsilon})-\varepsilon^{4},\
J_{2}(t,x;\alpha_{\varepsilon},\beta_{\varepsilon})\geq
J_{2}(t,x;\alpha_{\varepsilon},\beta)-\varepsilon^{4},
\mathbb{P}-a.s.,
\end{eqnarray}
and
\begin{eqnarray}\label{eqnarray1}
|\mathbb{E}[J_{j}(t,x;\alpha_{\varepsilon},\beta_{\varepsilon})]-e_{j}|\leq
\varepsilon^{4}, \ j=1,2.
\end{eqnarray}
Moreover, from Lemma \ref{l2} we know that  there exists a unique
couple $(u^{\varepsilon},v^{\varepsilon})$ such that
\begin{eqnarray*}\label{}
\alpha_{\varepsilon}(v^{\varepsilon})=u^{\varepsilon},
\beta_{\varepsilon}(u^{\varepsilon})=v^{\varepsilon}.
\end{eqnarray*}
Let us argue by contradiction. For this we observe that
(\ref{eqnarray1}) means  that (\ref{eq7}) holds. Assuming that
(\ref{eq6}) doesn't hold true, we have, for all $\varepsilon'> 0,$
the existence of some $\varepsilon\in (0,\varepsilon')$ (for which
we use the notations introduced above) and  $\delta\in[0,T-t]$  such
that, for some $j\in\{1,2\}$, say for $j=1$,
\begin{eqnarray}\label{e5}
\mathbb{P}\Big(\mathbb{P}(\
^{1}Y^{t,x;u^{\varepsilon},v^{\varepsilon}}_{t+\delta}<
W_{1}(t+\delta,X^{t,x;u^{\varepsilon},v^{\varepsilon}}_{t+\delta})-\varepsilon\
|\ \mathcal {F}_{t})> \varepsilon\Big)>0.
\end{eqnarray}
Put
\begin{eqnarray}\label{e3}
A=\Big\{\ ^{1}Y^{t,x;u^{\varepsilon},v^{\varepsilon}}_{t+\delta}<
W_{1}(t+\delta,X^{t,x;u^{\varepsilon},v^{\varepsilon}}_{t+\delta})-\varepsilon\
\Big\}\in\mathcal {F}_{t+\delta}.
\end{eqnarray}
By applying  Lemma \ref{le1} to $u^{\varepsilon}$ and $t+\delta$ we
see that,  there exists an NAD strategy $\widetilde{\alpha} \in
\mathcal {A}_{t,T},$ such that, for all $v\in \mathcal {V}_{t,T}$,
\begin{eqnarray}\label{eq18}
\widetilde{\alpha}(v)&=&u^{\varepsilon}, \text{on}\ [t,t+\delta],\nonumber\\
^{1}Y^{t,x;\widetilde{\alpha}(v),v}_{t+\delta}&\geq&
W_{1}(t+\delta,X^{t,x;\widetilde{\alpha}(v),v}_{t+\delta})-\frac{\varepsilon}{2},\
\mathbb{P}-a.s.
\end{eqnarray}
By virtue of Lemma \ref{l2} there exists a unique couple $(u,v)\in
\mathcal {U}_{t,T}\times \mathcal {V}_{t,T}$ such that
\begin{eqnarray*}\label{}
\widetilde{\alpha}(v)=u,\quad  \beta_{\varepsilon}(u)=v.
\end{eqnarray*}
We observe that this, in particular, means that $u=u^{\varepsilon}$
on $[t,t+\delta]$.  Let us define now a control $\widetilde{u}\in
\mathcal {U}_{t,T}$ as follows:
\begin{equation*}\label{}
\widetilde{u}=\left\{
\begin{array}{ll}
 u^{\varepsilon}, &\ \text{on}\quad ([t,t+\delta)\times\Omega)\cup([t+\delta,T]\times A^{c}),\\
u, & \ \text{on}\quad  [t+\delta,T]\times A.
\end{array}
\right.
\end{equation*}
Since $ \beta_{\varepsilon}\in\mathcal {B}_{t,T}$ is nonanticipative
it follows that
 $\beta_{\varepsilon}(\widetilde{u})=\beta_{\varepsilon}(u^{\varepsilon})=v^{\varepsilon}$ on
$[t,t+\delta]$, and for all $s\in[t+\delta,T]$,
\begin{equation*}\label{}
\beta_{\varepsilon}(\widetilde{u})_{s}=\left\{
\begin{array}{ll}
 \beta_{\varepsilon}(u)_{s}=v_{s}, & \text{on}\ A,\\
\beta_{\varepsilon}(u^{\varepsilon})_{s}=v^{\varepsilon}_{s}, &
\text{on}\ A^{c}.
\end{array}
\right.
\end{equation*}
Then  we have
\begin{equation*}\label{}
X^{t,x;\widetilde{u},\beta_{\varepsilon}(\widetilde{u})}=X^{t,x;u^{\varepsilon},v^{\varepsilon}},\
\text{on}\ [t,t+\delta],
\end{equation*}
\begin{equation*}\label{}
X^{t,x;\widetilde{u},\beta_{\varepsilon}(\widetilde{u})}=\left\{
\begin{array}{ll}
 X^{t,x;\widetilde{\alpha}(v),v}, & \text{on}\ [t+\delta,T]\times A,\\
X^{t,x;u^{\varepsilon},v^{\varepsilon}}, & \text{on}\
[t+\delta,T]\times A^{c},
\end{array}
\right.
\end{equation*}
and standard arguments show that also
\begin{equation*}\label{}
^{1}Y^{t,x;\widetilde{u},\beta_{\varepsilon}(\widetilde{u})}_{t+\delta}=\left\{
\begin{array}{ll}
 ^{1}Y^{t,x;\widetilde{\alpha}(v),v}_{t+\delta}, & \text{on}\  A,\\
^{1}Y^{t,x;u^{\varepsilon},v^{\varepsilon}}_{t+\delta}, & \text{on}\
 A^{c}.
\end{array}
\right.
\end{equation*}
Therefore,
\begin{eqnarray*}\label{}
 J_{1}(t,x;\widetilde{u},\beta_{\varepsilon}(\widetilde{u}))&=&\ ^
 {1}Y^{t,x;\widetilde{u},\beta_{\varepsilon}(\widetilde{u})}_{t}=
 \ ^{1}G^{t,x;\widetilde{u},\beta_{\varepsilon}(\widetilde{u})}_{t,t+\delta}
 [^{1}Y^{t,x;\widetilde{u},\beta_{\varepsilon}(\widetilde{u})}_{t+\delta}]\\
&=&^{1}G^{t,x;\widetilde{u},\beta_{\varepsilon}(\widetilde{u})}_{t,t+\delta}
 [^{1}Y^{t,x;\widetilde{u},\beta_{\varepsilon}(\widetilde{u})}_{t+\delta}1_{A}
 +\ ^{1}Y^{t,x;\widetilde{u},\beta_{\varepsilon}(\widetilde{u})}_{t+\delta}1_{A^{c}}]\\
&=&^{1}G^{t,x;\widetilde{u},\beta_{\varepsilon}(\widetilde{u})}_{t,t+\delta}
 [ ^{1}Y^{t,x;\widetilde{\alpha}(v),v}_{t+\delta}1_{A}
 +\
 ^{1}Y^{t,x;u^{\varepsilon},v^{\varepsilon}}_{t+\delta}1_{A^{c}}].
\end{eqnarray*}
Thanks to Lemma \ref{l1} and (\ref{eq18}) we have
\begin{eqnarray*}\label{}
 J_{1}(t,x;\widetilde{u},\beta_{\varepsilon}(\widetilde{u}))
&=&^{1}G^{t,x;\widetilde{u},\beta_{\varepsilon}(\widetilde{u})}_{t,t+\delta}
 [ ^{1}Y^{t,x;\widetilde{\alpha}(v),v}_{t+\delta}1_{A}
 +\ ^{1}Y^{t,x;u^{\varepsilon},v^{\varepsilon}}_{t+\delta}1_{A^{c}}]\\
 &\geq&^{1}G^{t,x;\widetilde{u},\beta_{\varepsilon}(\widetilde{u})}_{t,t+\delta}
 [(W_{1}(t+\delta,X^{t,x;\widetilde{\alpha}(v),v}_{t+\delta})-\frac{\varepsilon}{2})1_{A}
 +\ ^{1}Y^{t,x;u^{\varepsilon},v^{\varepsilon}}_{t+\delta}1_{A^{c}}]\\
  &=&^{1}G^{t,x;\widetilde{u},\beta_{\varepsilon}(\widetilde{u})}_{t,t+\delta}
 [W_{1}(t+\delta,X^{t,x;u^{\varepsilon},v^{\varepsilon}}_{t+\delta})1_{A}
 +\ ^{1}Y^{t,x;u^{\varepsilon},v^{\varepsilon}}_{t+\delta}1_{A^{c}}-\frac{\varepsilon}{2}1_{A}].
\end{eqnarray*}
Hence, from (\ref{e3}) it follows that
\begin{eqnarray}\label{eqn6}
 J_{1}(t,x;\widetilde{u},\beta_{\varepsilon}(\widetilde{u}))
  &\geq&^{1}G^{t,x;\widetilde{u},\beta_{\varepsilon}(\widetilde{u})}_{t,t+\delta}
 [W_{1}(t+\delta,X^{t,x;u^{\varepsilon},v^{\varepsilon}}_{t+\delta})1_{A}
 +\ ^{1}Y^{t,x;u^{\varepsilon},v^{\varepsilon}}_{t+\delta}1_{A^{c}}-\frac{\varepsilon}{2}1_{A}]\nonumber\\
  &\geq&^{1}G^{t,x;\widetilde{u},\beta_{\varepsilon}(\widetilde{u})}_{t,t+\delta}
 [ (^{1}Y^{t,x;u^{\varepsilon},v^{\varepsilon}}_{t+\delta}+\varepsilon)1_{A}
 +\ ^{1}Y^{t,x;u^{\varepsilon},v^{\varepsilon}}_{t+\delta}1_{A^{c}}-\frac{\varepsilon}{2}1_{A}]\nonumber\\
   &=&^{1}G^{t,x;\widetilde{u},\beta_{\varepsilon}(\widetilde{u})}_{t,t+\delta}
 [^{1}Y^{t,x;u^{\varepsilon},v^{\varepsilon}}_{t+\delta}+\frac{\varepsilon}{2}1_{A}]\nonumber\\
  &=&^{1}G^{t,x;u^{\varepsilon},v^{\varepsilon}}_{t,t+\delta}
 [^{1}Y^{t,x;u^{\varepsilon},v^{\varepsilon}}_{t+\delta}+\frac{\varepsilon}{2}1_{A}].
\end{eqnarray}
Let $$y_{s}=\
^{1}G^{t,x;u^{\varepsilon},v^{\varepsilon}}_{s,t+\delta}
 [\
 ^{1}Y^{t,x;u^{\varepsilon},v^{\varepsilon}}_{t+\delta}+\dfrac{\varepsilon}{2}1_{A}],
 s\in [t,t+\delta].$$
 This process is the solution of  the following BSDE:
\begin{equation*}\label{}
\begin{array}{lll}
y_s &=&
^{1}Y^{t,x;u^{\varepsilon},v^{\varepsilon}}_{t+\delta}+\dfrac{\varepsilon}{2}1_{A}
+\displaystyle \int_s^{t+\delta}
f_{1}(r,X^{t,x;u^{\varepsilon},v^{\varepsilon}}_r,
y_r,z_r,u^{\varepsilon}_{r},v^{\varepsilon}_{r})dr\\&&\qquad-\displaystyle
\int_s^{t+\delta} z_rdB_r, \qquad s\in [t,t+\delta],
\end{array}
\end{equation*}
which we compare with
\begin{equation*}\label{}
\begin{array}{lll}
^{1}Y^{t,x;u^{\varepsilon},v^{\varepsilon}}_s &=& \
^{1}Y^{t,x;u^{\varepsilon},v^{\varepsilon}}_{t+\delta}
+\displaystyle \int_s^{t+\delta}
f_{1}(r,X^{t,x;u^{\varepsilon},v^{\varepsilon}}_r,\
^{1}Y^{t,x;u^{\varepsilon},v^{\varepsilon}}_r,
\ ^{1}Z^{t,x;u^{\varepsilon},v^{\varepsilon}}_r,u^{\varepsilon}_r, v^{\varepsilon}_r)dr\\
&&\qquad -\displaystyle \int_s^{t+\delta} \
^{1}Z^{t,x;u^{\varepsilon},v^{\varepsilon}}_rdB_r, \quad s\in
[t,t+\delta].
\end{array}
\end{equation*}
Putting
$$\overline{y}_{s}=y_{s}-\ ^{1}Y^{t,x;u^{\varepsilon},v^{\varepsilon}}_{s},\
 \overline{z}_{s}=z_{s}-\ ^{1}Z^{t,x;u^{\varepsilon},v^{\varepsilon}}_{s}, \quad s\in
[t,t+\delta],$$ we have
\begin{eqnarray}\label{e4}
\overline{y}_s &=&\displaystyle \int_s^{t+\delta}
[f_{1}(r,X^{t,x;u^{\varepsilon},v^{\varepsilon}}_r,
y_r,z_r,u^{\varepsilon}_{r},v^{\varepsilon}_{r})-f_{1}(r,X^{t,x;u^{\varepsilon},v^{\varepsilon}}_r,
\ ^{1}Y^{t,x;u^{\varepsilon},v^{\varepsilon}}_r, \
^{1}Z^{t,x;u^{\varepsilon},v^{\varepsilon}}_r,u^{\varepsilon}_r,
v^{\varepsilon}_r)]dr\nonumber\\&&+ \dfrac{\varepsilon}{2}1_{A}
-\displaystyle \int_s^{t+\delta} \overline{z}_rdB_r, \quad s\in
[t,t+\delta].
\end{eqnarray}
For notational simplification let us assume that the Brownian motion
$B$ is one dimensional, and we introduce, for $r\in [t,t+\delta]$,
\begin{eqnarray*}
&&a_{r}=1_{\{\overline{y}_{r}\neq
0\}}(\overline{y}_{r})^{-1}\Big(f_{1}(r,X^{t,x;u^{\varepsilon},v^{\varepsilon}}_r,
y_r,z_r,u^{\varepsilon}_{r},v^{\varepsilon}_{r})-f_{1}(r,X^{t,x;u^{\varepsilon},v^{\varepsilon}}_r,
\ ^{1}Y^{t,x;u^{\varepsilon},v^{\varepsilon}}_r,
z_r,u^{\varepsilon}_r,v^{\varepsilon}_r)\Big),\\
&&b_{r}=1_{\{\overline{z}_{r}\neq
0\}}(\overline{z}_{r})^{-1}\Big(f_{1}(r,X^{t,x;u^{\varepsilon},v^{\varepsilon}}_r,
\ ^{1}Y^{t,x;u^{\varepsilon},v^{\varepsilon}}_r,
z_r,u^{\varepsilon}_r,v^{\varepsilon}_r)\\&&\qquad \qquad \qquad
\qquad \qquad \qquad  \qquad
-f_{1}(r,X^{t,x;u^{\varepsilon},v^{\varepsilon}}_r, \
^{1}Y^{t,x;u^{\varepsilon},v^{\varepsilon}}_r, \
^{1}Z^{t,x;u^{\varepsilon},v^{\varepsilon}}_r,u^{\varepsilon}_r,
v^{\varepsilon}_r)\Big).
   \end{eqnarray*}
Then, from the Lipschitz property of $f_{1}$ we see that
$|a_{r}|\leq L $, $|b_{r}|\leq L, r\in[t,t+\delta]$, and BSDE
(\ref{e4}) takes the following form:
\begin{eqnarray*}\label{}
\overline{y}_s &=& \dfrac{\varepsilon}{2}1_{A} +\displaystyle
\int_s^{t+\delta}
[a_{r}\overline{y}_r+b_{r}\overline{z}_r]dr\nonumber-\displaystyle
\int_s^{t+\delta} \overline{z}_rdB_r, s\in[t,t+\delta].
\end{eqnarray*}
By putting
$$Q_{s}=\exp\Big(\int_{t}^{s}a_{r}dr
  +\int_{t}^{s}b_{r}dB_{r}-\frac{1}{2}\int_{t}^{s}|b_{r}|^{2}dr\Big), s\in [t,t+\delta],$$
applying It\^o's formula  to $\overline{y}_s Q_{s}$, and then taking
the conditional   expectation, we deduce that
\begin{eqnarray*}
\overline{y}_t&=&\dfrac{\varepsilon}{2}\mathbb{E}[1_{A}Q_{t+\delta}\
|\mathcal{F}_{t}].
\end{eqnarray*}
By the Schwarz inequality we have
\begin{eqnarray*}
\mathbb{P}(A |\mathcal{F}_{t})^{2}=(\mathbb{E}[1_{A}|\mathcal
  {F}_{t}])^{2}\leq \mathbb{E}[1_{A}Q_{t+\delta}\
|\mathcal{F}_{t}]\mathbb{E}[Q_{t+\delta}^{-1}\ |\mathcal{F}_{t}].
\end{eqnarray*}
We observe that
\begin{eqnarray*}
\mathbb{E}[Q_{t+\delta}^{-1}\
|\mathcal{F}_{t}]&=&\mathbb{E}[\exp\Big(-\int_{t}^{t+\delta}a_{r}dr
 -\int_{t}^{t+\delta}b_{r}dB_{r}+\frac{1}{2}\int_{t}^{t+\delta}|b_{r}|^{2}dr\Big)\
 |\ \mathcal{F}_{t}]\\
&\leq& \exp(L\delta+L^{2}\delta)\mathbb{E}[\exp\Big(
 -\int_{t}^{t+\delta}b_{r}dB_{r}-\frac{1}{2}\int_{t}^{t+\delta}|b_{r}|^{2}dr\Big)\
 |\ \mathcal{F}_{t}]\\
 &=& \exp(L\delta+L^{2}\delta).
\end{eqnarray*}
 Let $$\Delta=\Big\{\mathbb{P}(\
^{1}Y^{t,x;u^{\varepsilon},v^{\varepsilon}}_{t+\delta}<
W_{1}(t+\delta,X^{t,x;u^{\varepsilon},v^{\varepsilon}}_{t+\delta})-\varepsilon\
|\ \mathcal {F}_{t})> \varepsilon\Big\}\Big(=\{\mathbb{P}(A
|\mathcal{F}_{t})>\varepsilon\}\Big).$$ Then,
\begin{eqnarray}\label{eqn7}
\overline{y}_t&=&\dfrac{\varepsilon}{2}\mathbb{E}[1_{A}Q_{t+\delta}\
|\mathcal{F}_{t}]\nonumber\\
&\geq&\dfrac{\exp(-L\delta-L^{2}\delta)\varepsilon}{2}(\mathbb{E}[1_{A}|\mathcal
  {F}_{t}])^{2}=\dfrac{\exp(-L\delta-L^{2}\delta)\varepsilon}{2}(\mathbb{P}(A|\mathcal
  {F}_{t}))^{2}\nonumber\\
&>&\dfrac{\varepsilon^{3}}{2}C_{0}1_{\Delta},
\end{eqnarray}
for $C_{0}=\exp(-L\delta-L^{2}\delta)$, where we use (\ref{e5}) in
the last inequality. Combining  (\ref{eqn7}) with
\begin{eqnarray*}\label{}
\overline{y}_{t}=y_{t}-\
^{1}Y_{t}^{t,x;u^{\varepsilon},v^{\varepsilon}}= \
^{1}G^{t,x;u^{\varepsilon},v^{\varepsilon}}_{t,t+\delta}
 [^{1}Y^{t,x;u^{\varepsilon},v^{\varepsilon}}_{t+\delta}+\frac{\varepsilon}{2}1_{A}]
-\ ^{1}G^{t,x;u^{\varepsilon},v^{\varepsilon}}_{t,t+\delta}
 [^{1}Y^{t,x;u^{\varepsilon},v^{\varepsilon}}_{t+\delta}],
\end{eqnarray*}
 we have
\begin{eqnarray*}\label{}
^{1}G^{t,x;u^{\varepsilon},v^{\varepsilon}}_{t,t+\delta}
 [^{1}Y^{t,x;u^{\varepsilon},v^{\varepsilon}}_{t+\delta}+\frac{\varepsilon}{2}1_{A}]
> \ ^{1}G^{t,x;u^{\varepsilon},v^{\varepsilon}}_{t,t+\delta}
 [^{1}Y^{t,x;u^{\varepsilon},v^{\varepsilon}}_{t+\delta}]+\dfrac{\varepsilon^{3}}{2}C_{0}1_{\Delta},
\end{eqnarray*}
and  (\ref{eqn6}) then yields
\begin{eqnarray*}\label{}
 J_{1}(t,x;\widetilde{u},\beta_{\varepsilon}(\widetilde{u}))
  &>&  J_{1}(t,x;\alpha_{\varepsilon},\beta_{\varepsilon})+\dfrac{\varepsilon^{3}}{2}C_{0}1_{\Delta}.
\end{eqnarray*}
 We can choose $\varepsilon'\in(0,1)$ sufficiently small
such that $\dfrac{\varepsilon'^{3}}{2}C_{0}>\varepsilon'^{4}$
(Recall that $\varepsilon'>0$ has been introduced at the beginning
of the proof, assuming  that (\ref{eq6}) doesn't hold true). Then
this relation is also satisfied by $\varepsilon\in (0,
\varepsilon'):$ $\dfrac{\varepsilon^{3}}{2}C_{0}>\varepsilon^{4}$.
Since $\mathbb{P}(\Delta)>0$, the above inequality  contradicts with
(\ref{e9}) for $\alpha(\cdot)=\widetilde{u}$. The proof is complete.
\end{proof}\\

We now give the existence theorem of a Nash equilibrium payoff.
\begin{theorem}\label{t2}
Let the Isaacs condition  (\ref{Isaacs}) hold. Then for all
$(t,x)\in[0,T]\times\mathbb{R}^{n}$, there exists a Nash equilibrium
payoff at $(t,x)$.
\end{theorem}
Let us admit the following Proposition for the moment. We shall give
its proof after.
\begin{proposition}\label{p2}
Under the assumptions of Theorem \ref{t1},  for all $\varepsilon>0,$
there exists
$(u^{\varepsilon},v^{\varepsilon})\in\mathcal{U}_{t,T}\times\mathcal{V}_{t,T}$
independent of $\mathcal {F}_{t}$ such that, for all $t\leq
s_{1}\leq s_{2}\leq T$, $j=1,2$,
\begin{eqnarray*}\label{}
\mathbb{P}\Big(\
 W_{j}(s_{1},X^{t,x;u^{\varepsilon},v^{\varepsilon}}_{s_{1}})-\varepsilon\leq
\ ^{j}G^{t,x;u^{\varepsilon},v^{\varepsilon}}_{s_{1},s_{2}}
 [W_{j}(s_{2},X^{t,x;u^{\varepsilon},v^{\varepsilon}}_{s_{2}})]\Big)> 1- \varepsilon.
\end{eqnarray*}
\end{proposition}

Let us begin with the proof  of Theorem \ref{t2}.\vskip2mm

\begin{proof}\label{}
By Theorem \ref{t1} we only have to prove that, for all
$\varepsilon>0,$ there exists
$(u^{\varepsilon},v^{\varepsilon})\in\mathcal{U}_{t,T}\times\mathcal{V}_{t,T}$
which satisfies (\ref{eq6}) and (\ref{eq7}) for $s\in[t,T], j=1,2$.

For  $\varepsilon>0$, we consider
$(u^{\varepsilon},v^{\varepsilon})\in\mathcal{U}_{t,T}\times\mathcal{V}_{t,T}$
given by Proposition \ref{p2}, i.e., in particular,
$(u^{\varepsilon},v^{\varepsilon})$ is independent of $\mathcal
{F}_{t}$, and we put $s_{1}=s, s_{2}=T$ in Proposition \ref{p2}.
This yields (\ref{eq6}). We also observe that the fact that
$(u^{\varepsilon},v^{\varepsilon})$ is independent of $\mathcal
{F}_{t}$ implies that $J_{j}(t,x;u^{\varepsilon},v^{\varepsilon}),
j=1,2,$ are deterministic and
$\Big\{(J_{1}(t,x;u^{\varepsilon},v^{\varepsilon}),J_{2}(t,x;u^{\varepsilon},v^{\varepsilon})),
\varepsilon>0\Big\}$ is a bounded sequence. Consequently, we can
choose an accumulation point of this sequence, as
$\varepsilon\rightarrow 0$. Let us denote this point by
$(e_{1},e_{2})$. Obviously,  this combined with (\ref{eq6}) allows
to conclude from Theorem \ref{t1} that $(e_{1},e_{2})$ is a Nash
equilibrium payoff at $(t,x)$. We also refer to the fact that since
$(u^{\varepsilon},v^{\varepsilon})$ is independent of $\mathcal
{F}_{t}$, the conditional probability $\mathbb{P}(\cdot |\mathcal
{F}_{t})$ of the event
$\Big\{W_{j}(s_{1},X^{t,x;u^{\varepsilon},v^{\varepsilon}}_{s_{1}})-\varepsilon\leq
\ ^{j}G^{t,x;u^{\varepsilon},v^{\varepsilon}}_{s_{1},s_{2}}
 [W_{j}(s_{2},X^{t,x;u^{\varepsilon},v^{\varepsilon}}_{s_{2}})]\Big\}$
coincides with its probability. Indeed, also
$\Big\{W_{j}(s_{1},X^{t,x;u^{\varepsilon},v^{\varepsilon}}_{s_{1}})-\varepsilon\leq
\ ^{j}G^{t,x;u^{\varepsilon},v^{\varepsilon}}_{s_{1},s_{2}}
 [W_{j}(s_{2},X^{t,x;u^{\varepsilon},v^{\varepsilon}}_{s_{2}})]\Big\}$ is independent of $\mathcal {F}_{t}$.
 The proof is complete.
\end{proof}\\

Before we present the proof of the above Proposition, we give the
following Lemmas, which will be needed.

\begin{lemma}
For all $\varepsilon>0,$ and all $\delta\in[0,T-t]$ and
$x\in\mathbb{R}^{n}$, there exists
$(u^{\varepsilon},v^{\varepsilon})\in\mathcal{U}_{t,T}\times\mathcal{V}_{t,T}$
independent of $\mathcal {F}_{t}$, such that
\begin{eqnarray*}\label{}
 W_{1}(t,x)-\varepsilon\leq
 \ ^{1}G^{t,x;u^{\varepsilon},v^{\varepsilon}}_{t,t+\delta}
 [W_{1}(t+\delta,X^{t,x;u^{\varepsilon},v^{\varepsilon}}_{t+\delta})], \ \mathbb{P}- a.s.,
\end{eqnarray*}
and
\begin{eqnarray*}\label{}
 W_{2}(t,x)-\varepsilon\leq
\ ^{2}G^{t,x;u^{\varepsilon},v^{\varepsilon}}_{t,t+\delta}
 [W_{2}(t+\delta,X^{t,x;u^{\varepsilon},v^{\varepsilon}}_{t+\delta})],\ \mathbb{P}- a.s.
\end{eqnarray*}
\end{lemma}

\begin{proof}
Let $\mathbb{F}^{t}=(\mathcal {F}^{t}_{s})_{s\in [t,T]}$ denote the
filtration generated by  $(B_{s}-B_{t})_{s\in [t,T]}$ and augmented
by the $\mathbb{P}$-null sets. By $\mathcal {U}^{t}_{s,T}$ (resp.,
$\mathcal {V}^{t}_{s,T}$) we denote the set of
$\mathbb{F}^{t}$-adapted processes $\{u_{r}\}_{r\in[s,T]}$ (resp.,
$\{v_{r}\}_{r\in[s,T]}$) taking their values in $U$ (resp., $V$).
Moreover, let $\mathcal {A}^{t}_{s,T}$ (resp., $\mathcal
{B}^{t}_{s,T}$) denote the set of  NAD strategies which map from
$\mathcal {V}^{t}_{s,T}$ into $\mathcal {U}^{t}_{s,T}$ (resp.
$\mathcal {U}^{t}_{s,T}$ into $\mathcal {V}^{t}_{s,T}$). With this
setting we replace the  framework of SDEs driven by a  Brownian
motion $B=(B_{s})_{s\in [0,T]}$ by that of SDEs driven by  a
Brownian motion $(B_{s}-B_{t})_{s\in [t,T]}$. We also translate the
above arguments from the framework of SDEs to the associated BSDEs.
Then, proceeding in the same manner as above, but now in our new
framework, we have the Isaacs condition, for $j=1,2, s\in[t,T]$,
\begin{eqnarray*}\label{}
&&\sup\limits_{u\in U}\inf\limits_{v\in V}
\Big\{\frac{1}{2}tr(\sigma\sigma^{T}(s,x,u,v)A)+\langle p,
b(s,x,u,v)\rangle+f_{j}(s,x,y,p^{T}\sigma(s,x,u,v),u,v) \Big\}\\
&=&\inf\limits_{v\in V}\sup\limits_{u\in U}
\Big\{\frac{1}{2}tr(\sigma\sigma^{T}(s,x,u,v)A)+\langle p,
b(s,x,u,v)\rangle +f_{j}(s,x,y,p^{T}\sigma(s,x,u,v),u,v) \Big\}
\end{eqnarray*}
for the associated value functionals
\begin{eqnarray*}\label{}
\widetilde{W}_{1}(s,x)&=& \esssup_{\alpha\in\mathcal {A}^{t}_{s,T}}
\essinf_{\beta\in\mathcal {B}^{t}_{s,T}}
J_{1}(s,x;\alpha,\beta)=\essinf_{\beta\in\mathcal
{B}^{t}_{s,T}}\esssup_{\alpha\in\mathcal {A}^{t}_{s,T}}
J_{1}(s,x;\alpha,\beta),
\end{eqnarray*}
and
\begin{eqnarray*}\label{}
\widetilde{W}_{2}(s,x)&=&\essinf_{\alpha\in\mathcal
{A}^{t}_{s,T}}\esssup_{\beta\in\mathcal {B}^{t}_{s,T}}
J_{2}(s,x;\alpha,\beta)=\esssup_{\beta\in\mathcal
{B}^{t}_{s,T}}\essinf_{\alpha\in\mathcal {A}^{t}_{s,T}}
J_{2}(s,x;\alpha,\beta),
\end{eqnarray*}
$(s,x)\in[t,T]\times\mathbb{R}^{n}$.

For $j=1,2$, from \cite{BCQ} we know that $W_{j}$ restricted to
$[t,T]\times\mathbb{R}^{n}$ and $\widetilde{W}_{j}$ are inside the
class of continuous functions with at most polynomial growth and the
unique viscosity solutions of the same
Hamilton-Jacobi-Bellman-Isaacs equation. Consequently, they coincide
\begin{eqnarray*}\label{}
\widetilde{W}_{j}(s,x)=W_{j}(s,x), \
(s,x)\in[t,T]\times\mathbb{R}^{n}, j=1,2.
\end{eqnarray*}
From the dynamic programming principle for $\widetilde{W}_{j}$ and
by observing that $\mathcal {V}^{t}_{t,T}\subset \mathcal
{B}^{t}_{t,T}$ we have
\begin{eqnarray*}\label{}
 W_{1}(t,x)&=&\widetilde{W}_{1}(t,x)=\esssup_{\alpha\in\mathcal {A}^{t}_{t,T}}
\essinf_{\beta\in\mathcal {B}^{t}_{t,T}} \
^{1}G^{t,x;\alpha,\beta}_{t,t+\delta}
 [W_{1}(t+\delta,X^{t,x;\alpha,\beta}_{t+\delta})]\\
 &\leq&\esssup_{\alpha\in\mathcal {A}^{t}_{t,T}}
\essinf_{v\in\mathcal {V}^{t}_{t,T}} \
^{1}G^{t,x;\alpha(v),v}_{t,t+\delta}
 [W_{1}(t+\delta,X^{t,x;\alpha(v),v}_{t+\delta})].
\end{eqnarray*}
Consequently,  for $\varepsilon>0$ and $\delta>0$,  there exists
$\alpha_{\varepsilon}\in\mathcal{A}^{t}_{t,T}$ such that, for all
$v\in\mathcal{V}^{t}_{t,T}$,
\begin{eqnarray*}\label{}
 W_{1}(t,x)-\varepsilon\leq
\ ^{1}G^{t,x;\alpha_{\varepsilon}(v),v}_{t,t+\delta}
 [W_{1}(t+\delta,X^{t,x;\alpha_{\varepsilon}(v),v}_{t+\delta})],\
 \mathbb{P}-a.s.
\end{eqnarray*}
The symmetric argument allows to show that the existence of
$\beta_{\varepsilon}\in\mathcal{B}^{t}_{t,T}$ such that, for all
$u\in\mathcal{U}^{t}_{t,T}$,
\begin{eqnarray*}\label{}
 W_{2}(t,x)-\varepsilon\leq
\ ^{2}G^{t,x;u,\beta_{\varepsilon}(u)}_{t,t+\delta}
 [W_{2}(t+\delta,X^{t,x;u,\beta_{\varepsilon}(u)}_{t+\delta})],  \ \mathbb{P}-a.s.
\end{eqnarray*}
In the same way as shown in  Lemma \ref{l2}, we get the existence of
$(u^{\varepsilon},v^{\varepsilon})\in\mathcal{U}^{t}_{t,T}\times\mathcal{V}^{t}_{t,T}$
such that
\begin{eqnarray*}\label{}
\alpha_{\varepsilon}(v^{\varepsilon})=u^{\varepsilon},
\beta_{\varepsilon}(u^{\varepsilon})=v^{\varepsilon}.
\end{eqnarray*}
Therefore, we have
\begin{eqnarray*}\label{}
 W_{1}(t,x)-\varepsilon\leq
\ ^{1}G^{t,x;u^{\varepsilon},v^{\varepsilon}}_{t,t+\delta}
 [W_{1}(t+\delta,X^{t,x;u^{\varepsilon},v^{\varepsilon}}_{t+\delta})],
\end{eqnarray*}
and
\begin{eqnarray*}\label{}
 W_{2}(t,x)-\varepsilon\leq
\ ^{2}G^{t,x;u^{\varepsilon},v^{\varepsilon}}_{t,t+\delta}
 [W_{2}(t+\delta,X^{t,x;u^{\varepsilon},v^{\varepsilon}}_{t+\delta})].
\end{eqnarray*}
The proof is complete.
\end{proof}\\

We also need the following Lemma.

\begin{lemma}\label{le6}
Let $n\geq 1$ and let us fix some partition
$t=t_{0}<t_{1}<\cdots<t_{n}=T$ of the interval $[t,T]$. Then, for
all $\varepsilon>0,$  there exists
$(u^{\varepsilon},v^{\varepsilon})\in\mathcal{U}_{t,T}\times\mathcal{V}_{t,T}$
independent of $\mathcal {F}_{t}$, such that, for all
$i=0,\cdots,n-1$,
\begin{eqnarray*}\label{}
 W_{1}(t_{i},X^{t,x;u^{\varepsilon},v^{\varepsilon}}_{t_{i}})-\varepsilon\leq
\ ^{1}G^{t,x;u^{\varepsilon},v^{\varepsilon}}_{t_{i},t_{i+1}}
 [W_{1}(t_{i+1},X^{t,x;u^{\varepsilon},v^{\varepsilon}}_{t_{i+1}})],\
 \mathbb{P}- a.s.,
\end{eqnarray*}
and
\begin{eqnarray*}\label{}
 W_{2}(t_{i},X^{t,x;u^{\varepsilon},v^{\varepsilon}}_{t_{i}})-\varepsilon\leq
\ ^{2}G^{t,x;u^{\varepsilon},v^{\varepsilon}}_{t_{i},t_{i+1}}
 [W_{2}(t_{i+1},X^{t,x;u^{\varepsilon},v^{\varepsilon}}_{t_{i+1}})],\
 \mathbb{P}- a.s.
\end{eqnarray*}
\end{lemma}

\begin{proof}
We shall give the proof by induction. By the above lemma, it is
obvious  for $i=0$. We now assume that
$(u^{\varepsilon},v^{\varepsilon})$ independent of $\mathcal
{F}_{t}$, is constructed on the interval $[t,t_{i})$ and we shall
define  it on $[t_{i},t_{i+1})$. From the above lemma it follows
that, for all  $y\in \mathbb{R}^{n}$, there exists
$(u^{y},v^{y})\in\mathcal{U}_{t_{i},T}\times\mathcal{V}_{t_{i},T}$
independent of $\mathcal {F}_{t}$, such that,
\begin{eqnarray}\label{eq10}
 W_{j}(t_{i},y)-\frac{\varepsilon}{2}\leq
\ ^{j}G^{t_{i},y;u^{y},v^{y}}_{t_{i},t_{i+1}}
 [W_{j}(t_{i+1},X^{t_{i},y;u^{y},v^{y}}_{t_{i+1}})], \ \mathbb{P}- a.s,
 j=1,2.
\end{eqnarray}
Let us fix arbitrarily $j=1,2.$  Moreover, for $y, z\in
\mathbb{R}^{n}$ and $s\in[t_{i},t_{i+1}]$, we put
$$y^{1}_{s}=\ ^{j}G^{t_{i},y;u^{y},v^{y}}_{s,t_{i+1}}
 [W_{j}(t_{i+1},X^{t_{i},y;u^{y},v^{y}}_{t_{i+1}})], \ \text{and}\ \
 y^{2}_{s}=\ ^{j}G^{t_{i},z;u^{y},v^{y}}_{s,t_{i+1}}
 [W_{j}(t_{i+1},X^{t_{i},z;u^{y},v^{y}}_{t_{i+1}})],$$
and we consider the BSDEs:
\begin{equation*}\label{}
\begin{array}{lll}
y^{1}_{s}&=& W_{j}(t_{i+1},X^{t_{i},y;u^{y},v^{y}}_{t_{i+1}})
+\displaystyle \int_{s}^{t_{i+1}} f_{j}(r,X^{t_{i},y;u^{y},v^{y}}_r,
y_r^{1},z^{1}_r,u^{y}_r, v^{y}_r)dr\\&&\qquad-\displaystyle
\int_{s}^{t_{i+1}} z^{1}_rdB_r, \quad s\in[t_{i},t_{i+1}],
\end{array}
\end{equation*}
and
\begin{equation*}\label{}
\begin{array}{lll}
y^{2}_{s}&=& W_{j}(t_{i+1},X^{t_{i},z;u^{y},v^{y}}_{t_{i+1}})
+\displaystyle \int_{s}^{t_{i+1}} f_{j}(r,X^{t_{i},z;u^{y},v^{y}}_r,
y_r^{2},z^{2}_r,u^{y}_r, v^{y}_r)dr\\&&\qquad-\displaystyle
\int_{s}^{t_{i+1}} z^{2}_rdB_r,  \quad s\in[t_{i},t_{i+1}].
\end{array}
\end{equation*}
By virtue of the Lemmas \ref{l3},  \ref{l4} and  \ref{le5} we have
\begin{eqnarray*}\label{}
&&|^{j}G^{t_{i},y;u^{y},v^{y}}_{t_{i},t_{i+1}}
 [W_{j}(t_{i+1},X^{t_{i},y;u^{y},v^{y}}_{t_{i+1}})]
-\ ^{j}G^{t_{i},z;u^{y},v^{y}}_{t_{i},t_{i+1}}
 [W_{j}(t_{i+1},X^{t_{i},z;u^{y},v^{y}}_{t_{i+1}})]|^{2}\\
&\leq&C\mathbb{E}[|W_{j}(t_{i+1},X^{t,y;u^{y},v^{y}}_{t_{i+1}})
-W_{j}(t_{i+1},X^{t,z;u^{y},v^{y}}_{t_{i+1}})|^{2}
\Big|\mathcal {F}_{t}]\\
&&+C\mathbb{E}[|\displaystyle \int_{t_{i}}^{t_{i+1}}
f_{j}(r,X^{t_{i},y;u^{y},v^{y}}_r, y_r^{1},z^{1}_r,u^{y}_r,
v^{y}_r)dr -\displaystyle \int_{t_{i}}^{t_{i+1}}
f_{j}(r,X^{t_{i},z;u^{y},v^{y}}_r, y_r^{1},z^{1}_r,u^{y}_r,
v^{y}_r)dr|^{2}
\Big|\mathcal {F}_{t}]\\
&\leq&C\mathbb{E}[|X^{t,y;u^{y},v^{y}}_{t_{i+1}}
-X^{t,z;u^{y},v^{y}}_{t_{i+1}}|^{2} \Big|\mathcal
{F}_{t}]+C\mathbb{E}[\displaystyle \int_{t_{i}}^{t_{i+1}}|
X^{t_{i},y;u^{y},v^{y}}_r-X^{t_{i},z;u^{y},v^{y}}_r |^{2}dr
\Big|\mathcal {F}_{t}]\\
 &\leq & C|y-z|^{2}.
\end{eqnarray*}
Therefore, by the above inequality, Lemma \ref{le5} and (\ref{eq10})
\begin{eqnarray*}\label{}
 W_{j}(t_{i},z)-\varepsilon&\leq&
 W_{j}(t_{i},y)-\varepsilon+C|y-z|\\
 &\leq& \ ^{j}G^{t_{i},y;u^{y},v^{y}}_{t_{i},t_{i+1}}
 [W_{j}(t_{i+1},X^{t_{i},y;u^{y},v^{y}}_{t_{i+1}})]-\frac{\varepsilon}{2}+C|y-z|\\
&\leq& \ ^{j}G^{t_{i},z;u^{y},v^{y}}_{t_{i},t_{i+1}}
 [W_{j}(t_{i+1},X^{t_{i},z;u^{y},v^{y}}_{t_{i+1}})]-\frac{\varepsilon}{2}+C|y-z|\\
 &\leq& \ ^{j}G^{t_{i},z;u^{y},v^{y}}_{t_{i},t_{i+1}}
 [W_{j}(t_{i+1},X^{t_{i},z;u^{y},v^{y}}_{t_{i+1}})], \ \mathbb{P}- a.s.,
\end{eqnarray*}
for $C|y-z|\leq \dfrac{\varepsilon}{2}$.\\

 Let $\{O_{i}\}_{i\geq1}\subset \mathcal {B}(\mathbb{R}^{n})$ be a partition of
$\mathbb{R}^{n}$ with $ diam(O_{i})<\dfrac{\varepsilon}{2C}$ and let
$y_{l}\in O_{l}$. Then, for $z\in O_{l}$,
\begin{eqnarray}\label{e6}
 W_{j}(t_{i},z)-\varepsilon\leq
\ ^{j}G^{t_{i},z;u^{y_{l}},v^{y_{l}}}_{t_{i},t_{i+1}}
 [W_{j}(t_{i+1},X^{t_{i},z;u^{y_{l}},v^{y_{l}}}_{t_{i+1}})], \ \mathbb{P}- a.s.,
\end{eqnarray}
and we define
\begin{eqnarray*}\label{}
u^{\varepsilon}=\sum\limits_{l\geq1}1_{O_{l}}(X^{t,x;u^{\varepsilon},v^{\varepsilon}}_{t_{i}})u^{y_{l}},\
v^{\varepsilon}=\sum\limits_{l\geq1}1_{O_{l}}(X^{t,x;u^{\varepsilon},v^{\varepsilon}}_{t_{i}})v^{y_{l}}.
\end{eqnarray*}
Therefore, we have
\begin{eqnarray*}\label{}
&& ^{j}G^{t,x;u^{\varepsilon},v^{\varepsilon}}_{t_{i},t_{i+1}}
 [W_{j}(t_{i+1},X^{t,x;u^{\varepsilon},v^{\varepsilon}}_{t_{i+1}})]\\
  &=& ^{j}G^{t_{i},X^{t,x;u^{\varepsilon},v^{\varepsilon}}_{t_{i}};u^{\varepsilon},v^{\varepsilon}}_{t_{i},t_{i+1}}
 [\sum\limits_{l\geq1}W_{j}(t_{i+1},X^{t_{i},X^{t,x;u^{\varepsilon},v^{\varepsilon}}_{t_{i}};
 u^{\varepsilon},v^{\varepsilon}}_{t_{i+1}})1_{O_{l}}(X^{t,x;u^{\varepsilon},v^{\varepsilon}}_{t_{i}})]\\
   &=& ^{j}G^{t_{i},X^{t,x;u^{\varepsilon},v^{\varepsilon}}_{t_{i}};u^{\varepsilon},v^{\varepsilon}}_{t_{i},t_{i+1}}
 [\sum\limits_{l\geq1}W_{j}(t_{i+1},X^{t_{i},X^{t,x;u^{\varepsilon},v^{\varepsilon}}_{t_{i}};
 u^{y_{l}},v^{y_{l}}}_{t_{i+1}})1_{O_{l}}(X^{t,x;u^{\varepsilon},v^{\varepsilon}}_{t_{i}})]\\
    &=& \sum\limits_{l\geq1}\ ^{j}G^{t_{i},X^{t,x;u^{\varepsilon},v^{\varepsilon}}_{t_{i}};
    u^{y_{l}},v^{y_{l}}}_{t_{i},t_{i+1}}
 [W_{j}(t_{i+1},X^{t_{i},X^{t,x;u^{\varepsilon},v^{\varepsilon}}_{t_{i}};
 u^{y_{l}},v^{y_{l}}}_{t_{i+1}})]1_{O_{l}}(X^{t,x;u^{\varepsilon},v^{\varepsilon}}_{t_{i}}).
\end{eqnarray*}
The latter relation follows from the uniqueness of solutions of
BSDEs. From (\ref{e6}) it follows that
\begin{eqnarray*}\label{}
&& ^{j}G^{t,x;u^{\varepsilon},v^{\varepsilon}}_{t_{i},t_{i+1}}
 [W_{j}(t_{i+1},X^{t,x;u^{\varepsilon},v^{\varepsilon}}_{t_{i+1}})]\\
&\geq& \sum\limits_{l\geq1}
[W_{j}(t_{i},X^{t,x;u^{y_{l}},v^{y_{l}}}_{t_{i}})-\varepsilon]
1_{O_{l}}(X^{t,x;u^{\varepsilon},v^{\varepsilon}}_{t_{i}})\\
&=& \sum\limits_{l\geq1}
W_{j}(t_{i},X^{t,x;u^{y_{l}},v^{y_{l}}}_{t_{i}})
1_{O_{l}}(X^{t,x;u^{\varepsilon},v^{\varepsilon}}_{t_{i}})-\varepsilon\\
  &=&
  W_{j}(t_{i},X^{t,x;u^{\varepsilon},v^{\varepsilon}}_{t_{i}})-\varepsilon.
\end{eqnarray*}
The proof is complete.
\end{proof}\\

Finally, we give the proof of Proposition \ref{p2}.\\

\begin{proof}
Let $t=t_{0}<t_{1}<\cdots<t_{n}=T$ be a partition of  $[t,T]$, and
  $\tau=\sup\limits_{i}(t_{i+1}-t_{i})$.   From Lemma
\ref{le5} it follows that, for all $j=1,2$, $0\leq k\leq n$,
$s\in[t_{k},t_{k+1})$ and $(u,v)\in\mathcal{U}_{t
,T}\times\mathcal{V}_{t,T}$,
\begin{eqnarray}\label{e8}
&&\mathbb{E}[|W_{j}(t_{k},X_{t_{k}}^{t,x;u,v})
-W_{j}(s,X^{t,x;u,v}_{s})|^{2}]\nonumber\\
 &\leq&2\mathbb{E}[|W_{j}(t_{k},X_{t_{k}}^{t,x;u,v})
-W_{j}(s,X^{t,x;u,v}_{t_{k}})|^{2}]\nonumber\\&&+2\mathbb{E}[|W_{j}(s,X_{t_{k}}^{t,x;u,v})
-W_{j}(s,X^{t,x;u,v}_{s})|^{2}]\nonumber\\
&\leq&C|s-t_{k}| (1+
\mathbb{E}[|X_{t_{k}}^{t,x;u,v}|^{2}])+C\mathbb{E}[|X_{t_{k}}^{t,x;u,v}
-X^{t,x;u,v}_{s}|^{2}]\nonumber\\
 &\leq& C\tau.
\end{eqnarray}
Here and after $C$ is a constant which may be different from line to
line.

 By virtue of Lemma \ref{le6} we let
$(u^{\varepsilon},v^{\varepsilon})\in\mathcal{U}_{t
,T}\times\mathcal{V}_{t,T}$ be defined as in  Lemma \ref{le6} for
$\varepsilon=\varepsilon_{0}$, where $\varepsilon_{0}>0$ will be
specified  later. Then, we have for all $i, 0\leq i \leq n$,
\begin{eqnarray*}\label{}
 W_{j}(t_{i},X^{t,x;u^{\varepsilon},v^{\varepsilon}}_{t_{i}})-\varepsilon_{0}\leq
 \ ^{j}G^{t,x;u^{\varepsilon},v^{\varepsilon}}_{t_{i},t_{i+1}}
 [W_{j}(t_{i+1},X^{t,x;u^{\varepsilon},v^{\varepsilon}}_{t_{i+1}})],\
 \mathbb{P}- a.s.
\end{eqnarray*}
For  $t\leq s_{1}\leq s_{2}\leq T$, we suppose, without loss of
generality, that $t_{i-1}\leq s_{1} \leq t_{i}$ and $t_{k}\leq s_{2}
\leq t_{k+1}$, for some $1\leq i<k\leq n-1$. Therefore, by the
Lemmas \ref{l1} and \ref{l3} we have
\begin{eqnarray*}\label{}
&& ^{j}G^{t,x;u^{\varepsilon},v^{\varepsilon}}_{t_{i},t_{k+1}}
 [W_{j}(t_{k+1},X^{t,x;u^{\varepsilon},v^{\varepsilon}}_{t_{k+1}})]\\
 &=& ^{j}G^{t,x;u^{\varepsilon},v^{\varepsilon}}_{t_{i},t_{k}}[^{j}G^{t,x;u^{\varepsilon},v^{\varepsilon}}_{t_{k},t_{k+1}}
 [W_{j}(t_{k+1},X^{t,x;u^{\varepsilon},v^{\varepsilon}}_{t_{k+1}})]]\\
  &\geq& ^{j}G^{t,x;u^{\varepsilon},v^{\varepsilon}}_{t_{i},t_{k}}
  [W_{j}(t_{k},X^{t,x;u^{\varepsilon},v^{\varepsilon}}_{t_{k}})-\varepsilon_{0}]\\
   &\geq& ^{j}G^{t,x;u^{\varepsilon},v^{\varepsilon}}_{t_{i},t_{k}}
  [W_{j}(t_{k},X^{t,x;u^{\varepsilon},v^{\varepsilon}}_{t_{k}})]-C\varepsilon_{0}\\
   &\geq&\cdots\geq \ ^{j}G^{t,x;u^{\varepsilon},v^{\varepsilon}}_{t_{i},t_{i+1}}
  [W_{j}(t_{i+1},X^{t,x;u^{\varepsilon},v^{\varepsilon}}_{t_{i+1}})]-C(k-i)\varepsilon_{0}\\
 &\geq&  W_{j}(t_{i},X^{t,x;u^{\varepsilon},v^{\varepsilon}}_{t_{i}})-C(k-i+1)\varepsilon_{0}.
\end{eqnarray*}
and the above inequality yields
\begin{eqnarray*}\label{}
&& ^{j}G^{t,x;u^{\varepsilon},v^{\varepsilon}}_{s_{1},t_{k+1}}
 [W_{j}(t_{k+1},X^{t,x;u^{\varepsilon},v^{\varepsilon}}_{t_{k+1}})]\\
 &=& ^{j}G^{t,x;u^{\varepsilon},v^{\varepsilon}}_{s_{1},t_{i}}
 [\ ^{j}G^{t,x;u^{\varepsilon},v^{\varepsilon}}_{t_{i},t_{k+1}}
 [W_{j}(t_{k+1},X^{t,x;u^{\varepsilon},v^{\varepsilon}}_{t_{k+1}})]]\\
 &\geq& ^{j}G^{t,x;u^{\varepsilon},v^{\varepsilon}}_{s_{1},t_{i}}
  [W_{j}(t_{i},X^{t,x;u^{\varepsilon},v^{\varepsilon}}_{t_{i}})-C(k-i+1)\varepsilon_{0}]\\
&\geq& ^{j}G^{t,x;u^{\varepsilon},v^{\varepsilon}}_{s_{1},t_{i}}
  [W_{j}(t_{i},X^{t,x;u^{\varepsilon},v^{\varepsilon}}_{t_{i}})]-C(k-i+1)\varepsilon_{0}\\
  &\geq& ^{j}G^{t,x;u^{\varepsilon},v^{\varepsilon}}_{s_{1},t_{i}}
  [W_{j}(t_{i},X^{t,x;u^{\varepsilon},v^{\varepsilon}}_{t_{i}})]-\frac{\varepsilon}{2},
\end{eqnarray*}
where we put $\varepsilon_{0}=\dfrac{\varepsilon}{2Cn}$. Let us put
\begin{eqnarray}\label{eq110}
 I_{1}&=& ^{j}G^{t,x;u^{\varepsilon},v^{\varepsilon}}_{s_{1},t_{k+1}}
 [W_{j}(t_{k+1},X^{t,x;u^{\varepsilon},v^{\varepsilon}}_{t_{k+1}})]-\
   ^{j}G^{t,x;u^{\varepsilon},v^{\varepsilon}}_{s_{1},t_{i}}
  [W_{j}(t_{i},X^{t,x;u^{\varepsilon},v^{\varepsilon}}_{t_{i}})]+\frac{\varepsilon}{2}\geq0,\nonumber\\
 I_{2}&=&  ^{j}G^{t,x;u^{\varepsilon},v^{\varepsilon}}_{s_{1},s_{2}}
 [W_{j}(s_{2},X^{t,x;u^{\varepsilon},v^{\varepsilon}}_{s_{2}})]
 -W_{j}(s_{1},X^{t,x;u^{\varepsilon},v^{\varepsilon}}_{s_{1}})+\frac{\varepsilon}{2}.
\end{eqnarray}
We assert that
\begin{eqnarray*}\label{}
\mathbb{E}[| I_{1} - I_{2}|^{2}] \leq C\tau.
\end{eqnarray*}
Indeed, setting
$$y_{s}=\
^{j}G^{t,x;u^{\varepsilon},v^{\varepsilon}}_{s,t_{i}}
  [W_{j}(t_{i},X^{t,x;u^{\varepsilon},v^{\varepsilon}}_{t_{i}})], s\in [s_{1},
  t_{i}],$$
   we have the associated BSDEs:
\begin{equation*}\label{}
\begin{array}{lll}
y_{s}&=&
W_{j}(t_{i},X^{t,x;u^{\varepsilon},v^{\varepsilon}}_{t_{i}})
+\displaystyle \int_{s}^{t_{i}}
f_{j}(r,X^{t,x;u^{\varepsilon},v^{\varepsilon}}_r,
y_r,z_r,u^{\varepsilon}_r,
v^{\varepsilon}_r)dr\\&&\qquad-\displaystyle \int_{s}^{t_{i}}
z_rdB_r, \quad s\in [s_{1}, t_{i}].
\end{array}
\end{equation*}
On the other hand, putting
\begin{equation*}\label{}
\begin{array}{lll}
y'_{s}=
W_{j}(s_{1},X^{t,x;u^{\varepsilon},v^{\varepsilon}}_{s_{1}}), \ s\in
[s_{1}, t_{i}],
\end{array}
\end{equation*}
we have by Lemma \ref{l3}
\begin{eqnarray*}
&&|^{j}G^{t,x;u^{\varepsilon},v^{\varepsilon}}_{s_{1},t_{i}}
  [W_{j}(t_{i},X^{t,x;u^{\varepsilon},v^{\varepsilon}}_{t_{i}})]
  -W_{j}(s_{1},X^{t,x;u^{\varepsilon},v^{\varepsilon}}_{s_{1}})|^{2}\\
&\leq &
C\mathbb{E}[|W_{j}(t_{i},X^{t,x;u^{\varepsilon},v^{\varepsilon}}_{t_{i}})-
W_{j}(s_{1},X^{t,x;u^{\varepsilon},v^{\varepsilon}}_{s_{1}})
|^2|\mathcal {F}_{s_{1}}] \\
& & + C \mathbb{E}[\int_{s_{1}}^{t_{i}}
|f_{j}(r,X^{t,x;u^{\varepsilon},v^{\varepsilon}}_r,
y_r,z_r,u^{\varepsilon}_r, v^{\varepsilon}_r) |^2|\mathcal
{F}_{s_{1}}].
\end{eqnarray*}
Therefore, from the boundedness of $f_{j}$ and the independence of
$\mathcal {F}_{t}$ of
$(u^{\varepsilon},v^{\varepsilon})\in\mathcal{U}_{t,T}\times\mathcal{V}_{t,T}$,
\begin{eqnarray*}
&&\mathbb{E}[|^{j}G^{t,x;u^{\varepsilon},v^{\varepsilon}}_{s_{1},t_{i}}
  [W_{j}(t_{i},X^{t,x;u^{\varepsilon},v^{\varepsilon}}_{t_{i}})]
  -W_{j}(s_{1},X^{t,x;u^{\varepsilon},v^{\varepsilon}}_{s_{1}})|^{2}|\mathcal {F}_{t}]\\
  &\leq &
C\mathbb{E}[|W_{j}(t_{i},X^{t,x;u^{\varepsilon},v^{\varepsilon}}_{t_{i}})-
W_{j}(s_{1},X^{t,x;u^{\varepsilon},v^{\varepsilon}}_{s_{1}})
|^2|\mathcal {F}_{t}] + C (t_{i}-s_{1})\\
 &= &
C\mathbb{E}[|W_{j}(t_{i},X^{t,x;u^{\varepsilon},v^{\varepsilon}}_{t_{i}})-
W_{j}(s_{1},X^{t,x;u^{\varepsilon},v^{\varepsilon}}_{s_{1}}) |^2] +
C (t_{i}-s_{1}).
\end{eqnarray*}
From  (\ref{e8}) it follows that
\begin{eqnarray}\label{e10}
\mathbb{E}[|^{j}G^{t,x;u^{\varepsilon},v^{\varepsilon}}_{s_{1},t_{i}}
  [W_{j}(t_{i},X^{t,x;u^{\varepsilon},v^{\varepsilon}}_{t_{i}})]
  -W_{j}(s_{1},X^{t,x;u^{\varepsilon},v^{\varepsilon}}_{s_{1}})|^{2}]\leq C\tau.
\end{eqnarray}
By a similar argument we have
\begin{eqnarray}\label{e11}
\mathbb{E}[|^{j}G^{t,x;u^{\varepsilon},v^{\varepsilon}}_{s_{2},t_{k+1}}
 [W_{j}(t_{k+1},X^{t,x;u^{\varepsilon},v^{\varepsilon}}_{t_{k+1}})]
-W_{j}(s_{2},X^{t,x;u^{\varepsilon},v^{\varepsilon}}_{s_{2}})|^{2}]\leq
C\tau.
\end{eqnarray}
For $s\in [s_{1},s_{2}]$ we let $$y^{1}_{s}=\
^{j}G^{t,x;u^{\varepsilon},v^{\varepsilon}}_{s,t_{k+1}}
 [W_{j}(t_{k+1},X^{t,x;u^{\varepsilon},v^{\varepsilon}}_{t_{k+1}})]
 =\ ^{j}G^{t,x;u^{\varepsilon},v^{\varepsilon}}_{s,s_{2}}
 [^{j}G^{t,x;u^{\varepsilon},v^{\varepsilon}}_{s_{2},t_{k+1}}
 [W_{j}(t_{k+1},X^{t,x;u^{\varepsilon},v^{\varepsilon}}_{t_{k+1}})]],$$
 and
  $$y^{2}_{s}=\ ^{j}G^{t,x;u^{\varepsilon},v^{\varepsilon}}_{s,s_{2}}
 [W_{j}(s_{2},X^{t,x;u^{\varepsilon},v^{\varepsilon}}_{s_{2}})],$$
and we consider the associated BSDEs:
\begin{equation*}\label{}
\begin{array}{lll}
y^{1}_{s}&=&
^{j}G^{t,x;u^{\varepsilon},v^{\varepsilon}}_{s_{2},t_{k+1}}
 [W_{j}(t_{k+1},X^{t,x;u^{\varepsilon},v^{\varepsilon}}_{t_{k+1}})]
+\displaystyle \int_{s}^{s_{2}}
f_{j}(r,X^{t,x;u^{\varepsilon},v^{\varepsilon}}_r,
y_r^{1},z^{1}_r,u^{\varepsilon}_r,
v^{\varepsilon}_r)dr\\&&\qquad-\displaystyle \int_{s}^{s_{2}}
z^{1}_rdB_r,\quad s\in [s_{1},s_{2}],
\end{array}
\end{equation*}
and
\begin{equation*}\label{}
\begin{array}{lll}
y^{2}_{s}&=&
W_{j}(s_{2},X^{t,x;u^{\varepsilon},v^{\varepsilon}}_{s_{2}})
+\displaystyle \int_{s}^{s_{2}}
f_{j}(r,X^{t,x;u^{\varepsilon},v^{\varepsilon}}_r,
y_r^{2},z^{2}_r,u^{\varepsilon}_r,
v^{\varepsilon}_r)dr\\&&\qquad-\displaystyle \int_{s}^{s_{2}}
z^{2}_rdB_r, \quad s\in [s_{1},s_{2}].
\end{array}
\end{equation*}
By virtue of the Lemmas \ref{l3} and \ref{l4} we have
\begin{eqnarray*}\label{}
&&|^{j}G^{t,x;u^{\varepsilon},v^{\varepsilon}}_{s_{1},t_{k+1}}
 [W_{j}(t_{k+1},X^{t,x;u^{\varepsilon},v^{\varepsilon}}_{t_{k+1}})]
-\ ^{j}G^{t,x;u^{\varepsilon},v^{\varepsilon}}_{s_{1},s_{2}}
 [W_{j}(s_{2},X^{t,x;u^{\varepsilon},v^{\varepsilon}}_{s_{2}})]|^{2}\\
&\leq&C\mathbb{E}[|^{j}G^{t,x;u^{\varepsilon},v^{\varepsilon}}_{s_{2},t_{k+1}}
 [W_{j}(t_{k+1},X^{t,x;u^{\varepsilon},v^{\varepsilon}}_{t_{k+1}})]
-W_{j}(s_{2},X^{t,x;u^{\varepsilon},v^{\varepsilon}}_{s_{2}})|^{2}
\Big|\mathcal {F}_{s_{1}}].
\end{eqnarray*}
Consequently, from (\ref{e11}) it follows that
\begin{eqnarray*}\label{}
\mathbb{E}[|^{j}G^{t,x;u^{\varepsilon},v^{\varepsilon}}_{s_{1},t_{k+1}}
 [W_{j}(t_{k+1},X^{t,x;u^{\varepsilon},v^{\varepsilon}}_{t_{k+1}})]
-\ ^{j}G^{t,x;u^{\varepsilon},v^{\varepsilon}}_{s_{1},s_{2}}
 [W_{j}(s_{2},X^{t,x;u^{\varepsilon},v^{\varepsilon}}_{s_{2}})]|^{2}]\leq C\tau.
\end{eqnarray*}
By the above inequality and (\ref{e10}) we get
\begin{eqnarray*}\label{}
\mathbb{E}[| I_{1} - I_{2}|^{2}] \leq C\tau.
\end{eqnarray*}
Consequently,
\begin{eqnarray*}\label{}
\mathbb{P}(I_{2}\leq - \frac{\varepsilon}{2}) \leq \mathbb{P}( |
I_{1} - I_{2}|\geq \frac{\varepsilon}{2}) \leq \frac{4\mathbb{E}[|
I_{1} - I_{2}|^{2}]}{\varepsilon^{2}}\leq
\frac{4C\tau}{\varepsilon^{2}}\leq\varepsilon,
\end{eqnarray*}
where we choose $\tau\leq\dfrac{\varepsilon^{3}}{4C}$, and from
(\ref{eq110}) it follows that
\begin{eqnarray*}\label{}
\mathbb{P}\Big(\
 W_{j}(s_{1},X^{t,x;u^{\varepsilon},v^{\varepsilon}}_{s_{1}})-\varepsilon\leq
\ ^{j}G^{t,x;u^{\varepsilon},v^{\varepsilon}}_{s_{1},s_{2}}
 [W_{j}(s_{2},X^{t,x;u^{\varepsilon},v^{\varepsilon}}_{s_{2}})]\
\Big)\geq 1- \varepsilon.
\end{eqnarray*}
 The proof is complete.
\end{proof}

\vspace{4mm}

\noindent{\bf Acknowledgements.}

\vspace{2mm}
 The  author  thanks  Prof. Rainer Buckdahn for his
careful reading, helpful discussions and  suggestions. The author
also thanks the editor and two anonymous referees for their helpful
suggestions. This  work is supported by the Young Scholar Award for
Doctoral Students of the Ministry of Education of China and the
Marie Curie Initial Training Network (PITN-GA-2008-213841).

\end{document}